\documentclass{amsart}

\usepackage{amssymb, amsmath}
\usepackage{mathrsfs}
\usepackage{amscd}
\usepackage{verbatim}
\usepackage{enumerate}
\usepackage{a4wide}

\usepackage{hyperref}

\allowdisplaybreaks

\theoremstyle{plain}
\newtheorem{thrm}{Theorem}[section]
\theoremstyle{remark}
\newtheorem{rmrk}[thrm]{Remark}

\theoremstyle{plain}
\newtheorem{clry}[thrm]{Corollary}
\newtheorem{lemma}[thrm]{Lemma}
\newtheorem{prop}[thrm]{Proposition}
\newtheorem{defn}[thrm]{Definition}

\newtheorem{facts}[thrm]{Facts}

\numberwithin{equation}{section}

\newcommand{\inprod}[1]{\left\langle #1 \right\rangle}

\newcommand{\leaderfill}{\leaders\hbox to 1em{\hss.\hss}\hfill}
\newcommand{\RR}{\mathbb{R}}
\newcommand{\R}{\mathbb{R}}

\newcommand{\NN}{\mathbb{N}}
\newcommand{\LL}{\mathbb{L}}

\renewcommand{\P}{{\mathbb P}}
\newcommand{\EE}{\mathbb{E}}
\newcommand{\E}{\mathbb{E}}
\newcommand{\FF}{\mathbb{F}}
\newcommand{\DD}{\mathbb{D}}
\newcommand{\D}{\mathbb{D}}
\newcommand{\e}{\varepsilon}
\newcommand{\one}{\mathbf{1}}

\newcommand{\calL}{{\mathscr L}}
\renewcommand{\O}{\Omega}

\newcommand{\F}{{\mathscr F}}

\newcommand{\g}{\gamma}
\newcommand{\lb}{\langle}
\newcommand{\rb}{\rangle}
\newcommand{\limn}{\lim_{n\to\infty}}
\newcommand{\limk}{\lim_{k\to\infty}}
\newcommand{\umd}{\textsc{umd} }

\begin{document}

\author{Matthijs Pronk and Mark Veraar}
\address{Delft Institute of Applied Mathematics\\
Delft University of Technology \\ 2600 GA Delft\\The
Netherlands} \email{Matthijs.Pronk@tudelft.nl}
\email{M.C.Veraar@tudelft.nl}

\title{Tools for Malliavin calculus in UMD Banach spaces}

\begin{abstract}
In this paper we study the Malliavin derivatives and Skorohod integrals for
processes taking values in an infinite dimensional space. Such results are motivated by their applications to SPDEs and in particular financial mathematics. Vector-valued Malliavin theory in Banach space $E$ is naturally restricted to spaces $E$ which have the so-called \umd property, which arises in harmonic analysis and stochastic integration theory. We provide several new results and tools for the Malliavin derivatives and Skorohod integrals in an infinite dimensional setting. In particular, we prove weak characterizations, a chain rule for Lipschitz functions, a sufficient condition for pathwise continuity and an It\^o formula for non-adapted processes.
\end{abstract}

\thanks{The first named author is supported by VICI subsidy 639.033.604
of the Netherlands Organisation for Scientific Research (NWO). The second author
was supported by VENI subsidy 639.031.930
of the Netherlands Organisation for Scientific Research (NWO)}




\maketitle

\section{Introduction}

Since the seminal paper \cite{Malliavin78}, Malliavin calculus has played an influential role in probability theory (see the monographs \cite{Bell,BGJ,Bism,Bog10,Boga-Gaussian,DPMal,Jan97,Mal97,Nualart2,Ocone,Wat} and references therein). In particular, it has played an important part in the study of stochastic (partial) differential equations (S(P)DE) and mathematical finance (see the monographs \cite{CarmTeh,Fil,Im08,MalTha,NOP,SanzSole} and references therein). For certain models in finance and SPDEs, Malliavin calculus and Skorohod integration can be applied in an infinite dimensional framework. In the setting of Hilbert space-valued stochastic processes details on this matter can be found in \cite{CarmTeh,Fil,GrPar,LeNu} and references therein. For Banach space-valued stochastic processes, there are geometric obstacles which have to be overcome in order to extend stochastic calculus to this setting.

In  \cite{NVW1} a new It\^o type integration theory for processes with values in a Banach space $E$ has been developed using earlier ideas from \cite{Ga1,MC}. The theory uses a geometric assumption on $E$, called the \umd-property, and it allows two-sided estimates for $L^p$-moments for stochastic integrals (see Theorem \ref{thm:stochintI} below). A deep result in the theory of \umd spaces is that a Banach space $E$ has the \umd property if and only if the Hilbert transform is bounded on $L^p(\R;E)$ (see \cite{Bu3} and references therein).
The class of \umd spaces include all Hilbert spaces, $L^q$-spaces with $q\in (1, \infty)$ and the reflexive Sobolev spaces, Besov spaces and Orlicz spaces. Among these spaces the $L^q$-spaces with $q\in (1, \infty)$ are the most important ones for applications to SPDEs.
Recently, the full strength of the stochastic integration theory from \cite{NVW1} has made it possible to obtain optimal space-time-regularity results for a large class of SPDEs (see \cite{NVW11eq,NVW10}).

In \cite{janjan}, Malliavin calculus and Skorohod integration have been studied in the Banach space-valued setting. In particular, the authors have shown that if the space $E$ has the \umd property, the Skorohod integral is an extension of the It\^o integral for processes with values in $E$ (see Theorem \ref{div-norm-estimate} below). The main result of \cite{janjan} is a Clark-Ocone representation formula for $E$-valued random variables, where again $E$ is a \umd space. A previous attempt to obtain this representation formula was been given in \cite{MayZak1}, but the proof contains a gap (see \cite{MayZak2}).
Further developments on Malliavin calculus have been made in \cite{janmaas} and in particular, the connection with Meyer's inequalities has been investigated. Here the \umd property is needed again in order to obtain the vector-valued analagon of the Meyer's inequalities (see \cite{janmaas} and \cite{PisierMeyer}). In particular, in \cite{janmaas} a vector-valued version of Meyer's inequalities for higher order derivatives has been proved. Also here there has been a previous attempt to show Meyer's inequalities for higher order derivatives in \umd spaces (see \cite{MalNu}), but unfortunately the proof contains a gap since an integral in \cite[Theorem 1.17]{MalNu} is not convergent.

In this paper we proceed with the development of Malliavin calculus in the \umd-valued setting. After recalling some prerequisites, in Section \ref{subsec:indweak} we will prove a weak characterization of the Malliavin derivative and extend the Meyers-Serrin result for Sobolev spaces to the setting of Gaussian Sobolev spaces. In Section \ref{sec:calculus} several calculus facts such as the product and chain rule will be obtained. Under additional geometric conditions, we further extend the chain rule to the case of Lipschitz function in Section \ref{subsec:Lip}. Some new results for the Skorohod integral are derived in Section \ref{sec:Skorohod}. In particular, pathwise properties of the Skorohod integral are studied in Section \ref{sec:stochintprocess}. In the final Section \ref{sec:Ito1} we prove a version of It\^o's formula in the non-adapted setting. The results of this paper are motivated by future applications to SPDEs, which will be presented in a forthcoming paper of the authors.

\medskip

\textbf{Acknowledgement.} The authors thank Markus Kunze, Jan Maas, Jan van Neerven and the anonymous referee for helpful discussion and comments.

\section{Preliminaries}
Below all vector spaces are assumed to be real. With minor modifications, most results can be extended to complex spaces as well.
For a parameter $t$, and real numbers $A$ and $B$, we write $A\lesssim_t B$ to indicate that there is a constant $c$ only depending on $t$ such that $A\leq c B$. Moreover, we write $A\eqsim_t B$ if both $A\lesssim_t B$ and $B\lesssim_t A$ hold.

\subsection{$\gamma$-radonifying operators}
Let $(\O,\F,\P)$ be a probability space, let $H$ be a separable Hilbert space and let $E$ be a Banach space.
In this section we will review some results about $\gamma$-radonifying operators. For a detailed overview we refer to \cite{DJT,KaWe,Neerven-Radon}.
For $h\in H, x\in E$, we denote by $h\otimes x$ the operator in $\calL(H,E)$ defined by \[(h\otimes x)h' := \inprod{h, h'}x, \qquad h'\in H.\]
For finite rank operators $\sum_{j=1}^n h_j \otimes x_j$, where the vectors $h_1, \ldots, h_n\in H$ are orthonormal and $x_1, \ldots, x_n\in E$, we define
\[ \Big\|\sum_{j=1}^n h_j \otimes x_j\Big\|_{\gamma(H,E)} := \Big(\EE\Big\| \sum_{n=1}^N \gamma_n x_n\Big\|^2_E\Big)^{1/2}.\]
Here $(\gamma_n)_{n\geq 1}$ is a sequence of independent standard real-valued Gaussian random variables.
The space $\gamma(H,E)$ of {\em $\gamma$-radonifying operators} is defined as the closure of all finite rank operators in the norm $\|\cdot\|_{\gamma(H,E)}$. One can show that a bounded operator $R:H\to E$ belongs to $\gamma(H,E)$ if and only if the sum $\sum_{n\geq 1} \gamma_n R h_n$ converges in $L^2(\Omega;E)$.

By the Kahane--Khintchine inequalities (see \cite[Corollary 3.2]{LeTa}) one has
\[\Big(\E\Big\| \sum_{n=1}^N \gamma_n x_n\Big\|^p\Big)^{1/p} \eqsim_p \Big(\E\Big\| \sum_{n=1}^N \gamma_n x_n\Big\|^2\Big)^{1/2},\]
and this extends to infinite sums as well, whenever the sum is convergent.

The $\gamma$-radonifying operators also satisfy the following ideal property (see \cite{Bax,LinPie,Neerven-Radon}).
\begin{prop}[Ideal property]\label{gamma-ideal-property}
Suppose that $H_0$ and $H_1$ are Hilbert spaces and $E_0$ and $E_1$ are Banach spaces.
Let $R\in \gamma(H_0,E_0)$, $T\in \calL(H_1,H_0)$ and $U\in \calL(E_0,E_1)$, then $URT\in \gamma(H_1,E_1)$ and \[\|URT\|_{\gamma(H_1,E_1)} \leq \|U\| \|R\|_{\gamma(H_0,E_0)} \|T\|.\]
\end{prop}

The following lemma is called the $\gamma$-Fubini Lemma, and is taken from \cite{NVW1}.
\begin{lemma}\label{sec:gamma,Fubini-lemma}
Let $(S, \Sigma, \mu)$ be a $\sigma$-finite measure space and let $1\leq p<\infty$. The mapping $U: L^p(S;\gamma(H,E)) \to \calL(H, L^p(S;E))$, given by $((UF)h)s = F(s)h$ for $s\in S$ and $h\in H$, defines an isomorphism $L^p(S;\gamma(H, E)) \eqsim \gamma(H, L^p(S;E))$.
\end{lemma}

The following lemma is taken from \cite{KaWe} and for convenience we include a proof.
\begin{lemma}\label{lem:trace}
Suppose that $H_0$ and $H_1$ are Hilbert spaces. Let $R\in \gamma(H_0,E)$ and $S\in \gamma(H_1,E^*)$. Define the operator $\inprod{R,S}\in \calL(H_1, H_0)$ by
\[\inprod{R,S}h := R^*(Sh),\ h\in H_2.\]
Then $\inprod{R,S}$ is a trace class operator.
\end{lemma}
\begin{proof}
Let $(h_n)_{n\geq 1}$ and $(k_n)_{n\geq 1}$ be orthonormal bases for $H_0$ and $H_1$, respectively. Let $(\varepsilon_n)_{n\geq 1}$ be such that $|\varepsilon_n|=1$ and $\lb R h_n, S u_n\rb \varepsilon_n = |\lb R h_n, S u_n\rb|$ for every $n\geq 1$. Then by H\"older's inequality
\begin{align*}
\sum_{n\geq 1} \Big|\Big<\inprod{R,S} u_n, h_n\Big>\Big| & = \sum_{n\geq 1} \lb \varepsilon_n R  h_n, S u_n\rb
\\ & = \EE\Big\langle \sum_{n\geq 1} \varepsilon_n \gamma_n R h_n, \sum_{k\geq 1} \gamma_k S u_k\Big\rangle
\\ & \leq \Big\|\sum_{n\geq 1} \varepsilon_n \gamma_n R  h_n\Big\|_{L^2(\O;E)} \Big\|\sum_{k\geq 1} \gamma_k S u_k\Big\|_{L^2(\O;E^*)}
\\ & = \|R\|_{\gamma(H_0,E)} \|S\|_{\gamma(H_1,E^*)}.
\end{align*}
Now the result follows from \cite[Theorem 4.6]{DJT}.
\end{proof}

We define the {\em trace duality pairing} as
\[\inprod{R, S}_{\textrm{Tr}} = \textrm{Tr}(R^* S) = \sum_{n\geq 1} \inprod{Rh_n, Su_n}.\]
From the calculation in the above proof we see that
$|\inprod{R, S}_{\textrm{Tr}}| \leq \|R\|_{\gamma(H,E)} \|S\|_{\gamma(H,E^*)}$.

Recall the following facts: (details and references on \umd and type can be found in Subsections \ref{subsec:Ornstein} and \ref{sec:stochintprocess}, respectively).
\begin{facts}\label{factsgamma} \
\begin{itemize}
\item If $E$ is a Hilbert space, then $\g(H,E)$ coincides with the Hilbert-Schmidt operators $S^2(H,E)$
\item For all $p\in [1, \infty)$, $\g(H,E)$ is isomorphic to a closed subspace of $L^p(\O;E)$.
\end{itemize}
Moreover, the following properties of $E$ are inherited by $\g(H,E)$: reflexivity, type $p\in [1, 2]$, cotype $q\in[2, \infty]$, \umd, separability.
\end{facts}

\subsection{The Malliavin derivative operator}
In this section we recall some of the basic elements of Malliavin calculus. We refer to \cite{Nualart2} for details in the scalar situation.

Let $\{W(h), h\in H\}$ be an {\em isonormal Gaussian process} associated with $H$, that is $\{W h:h\in H\}$ is a centered family of Gaussian random variable and
\[\E (Wh_1 W h_2) = \lb h_1, h_2\rb, \ \ \ \ \ h_1, h_2\in H.\]
We will assume $\mathscr{F}$ is generated by $W$.

Let $1\leq p<\infty$, and let $E$ be a Banach space.
Let us define the {\em Gaussian Sobolev space $\DD^{1,p}(E)$} of $E$-valued random variables in the following way. Consider the class $\mathscr{S}\otimes E$ of {\em smooth $E$-valued random variables} $F:\Omega\to E$ of the form
\[F = f(W(h_1), \ldots, W(h_n)) \otimes x,\]
where $f\in C^\infty_b(\RR^n)$, $h_1,\ldots, h_n\in H$, $x\in E$, and linear combinations thereof. Since $\mathscr{S}$ is dense in $L^p(\O)$ and $L^p(\O)\otimes E$ is dense in $L^p(\O;E)$, it follows that $\mathscr{S} \otimes E$ is dense in $L^p(\Omega;E)$.
For $F\in \mathscr{S}\otimes E$, define the {\em Malliavin derivative $DF$} as the random variable $DF:\Omega\to\gamma(H,E)$ given by \[DF = \sum_{i=1}^n \partial_i f(W(h_1),\ldots, W(h_n)) \otimes (h_i \otimes x).\]
If $E= \R$, we can identify $\g(H,\R)$ with $H$ and in that case for all $F\in \mathscr{S}$, $DF\in L^p(\O;H)$ coincides with the Malliavin derivative in \cite{Nualart2}. Recall from \cite[Proposition 1.2.1]{Nualart2} that $D$ is closable as an operator from $L^p(\O)$ into $L^p(\O;H)$, and this easily extends to the vector-valued setting (see \cite[Proposition $3.3$]{janjan}). For convenience we provide a short proof.

\begin{prop}[Closability]\label{closable}
For all $1\leq p < \infty$, the Malliavin derivative $D$ is closable as an operator from $L^p(\Omega;E)$ into $L^p(\Omega;\gamma(H,E))$.
\end{prop}
\begin{proof}
Let $(F_n)_{n\geq 1}$ in $\mathscr{S}\otimes E$ and $G\in L^p(\O;\g(H,E))$ be such that $\limn F_n = 0$ in $L^p(\O;X)$ and $\limn D F_n = G$ in $L^p(\O;\g(H,E))$. We need to show that $G =0$. Since $G$ is strongly measurable, it suffices to check that for any $h\in H$ and $x^*\in E^*$ one has $\lb G h, x^*\rb =0$. By the closability of $D$ in the scalar case one obtains
\[\lb G h, x^*\rb = \limn \lb D F_n h, x^*\rb = \limn D (\lb F_n, x^*\rb) (h) =0.\]
\end{proof}

The closure of the operator $D$ is denoted by $D$ again. The domain of the closure is denoted by $\DD^{1,p}(E)$ and endowed with the norm
\[ \|F\|_{\DD^{1,p}(E)} := (\|F\|^p_{L^p(\Omega;E)} + \|DF\|^p_{L^p(\Omega;\gamma(H,E))})^{1/p}\]
it becomes a Banach space. Similarly, for $k\geq 2$ and $p\geq 1$ we let $\D^{k,p}(E)$ be the closure of $\mathscr {S}\otimes E$ with respect to the norm
\[\|F\|_{\DD^{k,p}(E)} := (\|F\|^p_{L^p(\Omega;E)} + \sum_{i=1}^k\|D^iF\|^p_{L^p(\Omega;\gamma^i(H,E))})^{1/p}.\]
Here $\gamma^1(H,E) = \gamma(H, E)$ and recursively, $\gamma^n(H,E) = \gamma(H, \gamma^{n-1}(H,E))$ for $n\geq 2$.
Finally let $\DD^{\infty,p}(E) = \bigcap_{k\geq 1} \DD^{k,p}(E)$.

\subsection{Ornstein-Uhlenbeck operators and Meyer's inequalities\label{subsec:Ornstein}}
In this subsection we recall several results from \cite{janmaas}  and \cite{Nualart2}.

Recall the definition of the {\em $n$-th Wiener chaos}
\[ \mathcal{H}_n := \overline{\textrm{lin}}\{H_n(W(h)):\ \|h\|=1\}.\]
Here, $H_n$ is the $n$-th Hermite polynomial. Also let $\mathscr{P}$ be the set of random variable of the form $p(W(h_1), \ldots W(h_n))$ where $p$ is a polynomial and $h_1, \ldots, h_n\in H$. This set is dense in $L^p(\O)$ for all $p\in [1, \infty)$ (see \cite[Exercise 1.1.7]{Nualart2}).
A classical result is the following orthogonal decomposition $L^2(\Omega) = \bigoplus_{n\geq 0} \mathcal{H}_n$ (see \cite[Theorem 1.1.1]{Nualart2}).
For each $n\geq 1$, let $J_n\in \calL(L^2(\O))$ be the orthogonal projection onto $\mathcal{H}_n$.
The {\em Ornstein-Uhlenbeck semigroup} $(P(t))_{t\geq 0}$ on $L^2(\Omega)$ is defined by
\[P(t) := \sum_{n=0}^\infty e^{-nt} J_n.\]
Clearly, $P(t)^* = P(t)$.
Moreover, $(P(t))_{t\geq 0}$ extends to a $C_0$-semigroup of positive contractions on $L^p(\Omega)$ for all $1\leq p<\infty$ (see \cite[Section 1.4]{Nualart2}).

Let $E$ be a Banach space. By positivity, for any $t\geq0$ the mapping $P(t) \otimes I_E$ extends to a contraction $P_E(t)\in \calL(L^p(\O;E))$. Moreover, $(P_E(t))_{t\geq 0}$ on $L^p(\Omega;E)$ is a $C_0$-semigroup. Also note that $\mathscr{P}\otimes E$ is dense in $L^p(\O;E)$.
Denote the generator of $(P_E(t))_{t\geq 0}$ by $L_E$, and set $C_E := -\sqrt{-L_E}$.
Observe that for all $F\in \mathscr{P}\otimes E$ one has $C_E F = \sum_{n\geq 0} \sqrt{n} J_n F$. Whenever, there is no danger of confusion we will leave out the subscript $E$ from all these expressions.
If $E$ does not contain $\ell^1_n$'s uniformly, then $P_E(t)$ is an analytic semigroup on $L^p(\Omega;E)$ for all $p\in (1, \infty)$. In this case each $J_n$ is a bounded operator on $L^p(\O;E)$ for all $p\in (1, \infty)$. We refer to \cite[Theorem $5.5$, Remark $5.9(ii)$, Identity $(5.15)$]{Pisier} for details.

Next we recall the vector-valued analagon of Meyer's inequalities from \cite[Theorem 6.8]{janmaas}. To do so we need the following Banach space property.
A Banach space $E$ is said to have {\em \umd} if for
some $p\in (1, \infty)$, there is a constant $\beta_{p,E}$ such that for every $N\geq 1$, every martingale difference sequence\index{martingale
difference sequence} $(d_n)_{n=1}^N$ in $L^p(\Omega, E)$ and every $\{-1,
1\}$-valued sequence $(\e_n)_{n=1}^N$, we have
\begin{equation}\label{eq:UMD}
\Big(\E\Big\|\sum_{n=1}^N \e_n d_n \Big\|^p\Big)^\frac1p \leq \beta_{p,X}
\Big(\E\Big\|\sum_{n=1}^N d_n\Big\|^p\Big)^\frac1p.
\end{equation}
\umd stands for unconditional martingale differences. It can be shown that if \eqref{eq:UMD} holds for some $p\in (1, \infty)$, then one can show that it holds for all $p\in (1, \infty)$. We refer to \cite{Bu3} for details. The \umd property plays an important role in vector-valued harmonic analysis, due to the fact that the Hilbert transform is bounded if and only if $E$ is a \umd space.
We will also use that the \umd property implies several other useful Banach space properties.
If $X$ is a \umd space, then it is reflexive, and hence spaces such as $C(K)$, $L^1$, $L^\infty$ do not have \umd. In the reflexive range many of the classical spaces (Lebesgue space, Sobolev spaces, Besov spaces, Orlicz spaces, Schatten class, etc) are known to be \umd. In applications to SPDEs the most important example is $L^q$ with $q\in (1, \infty)$.

The case $n=1$ of the following result was proved in \cite{PisierMeyer} and used in \cite[Theorem 6.8]{janmaas} to derive the case $n\geq 2$ by induction.
\begin{thrm}[Meyer's inequalities]\label{Meyer's inequalities}
Let $E$ be a \umd Banach space, let $1<p<\infty$ and $n\geq 1$. Then the domain of the operator $C^n$ on $L^p(\Omega;E)$ equals $\DD^{n,p}(E)$. Moreover, for all $F\in \DD^{n,p}(E)$ we have
\begin{align*}
\|D^nF\|_{L^p(\Omega;\gamma^n(H,E))} &\lesssim_{p,E,n} \|C^n F\|_{L^p(\Omega;E)}
\lesssim_{p,E,n} \|F\|_{L^p(\Omega;E)} + \|D^nF\|_{L^p(\Omega;\gamma^n(H,E))}.
\end{align*}
\end{thrm}
Also recall the following vector-valued version of Meyer's Multiplier Theorem (see \cite[Theorem 6.5]{janmaas}, where even an operator-valued version has been obtained).
\begin{thrm}[Meyer's Multiplier Theorem]\label{Meyer's Multiplier Theorem}
Let $1<p<\infty$. Let $E$ be a \umd Banach space, and let $(a_k)_{k=0}^\infty$ be a sequence of real numbers such that $\sum_{k=0}^\infty |a_k|N^{-k} < \infty$ for some $N \geq 1$. If a sequence of scalars $(\phi(n))_{n\geq 1}$ satisfies $\phi(n) := \sum_{k=0}^\infty a_k n^{-k}$ for $n\geq N$, then the operator $T_\phi$ defined by
\[ T_\phi F := \sum_{n=0}^\infty \phi(n)J_nF, \qquad F\in \mathscr{P}\otimes E,\]
extends to a bounded operator on $L^p(\Omega;E)$.
\end{thrm}
Due to the above results, many of the results in the scalar setting can be extended to the \umd-valued setting. Some results have already been derived in \cite{janmaas}, and we will obtain several other results which will be needed to present some of the tools in Malliavin calculus in the \umd setting.

The following density result is a consequence of the corresponding result in the scalar case. It will play a minor role in the sequel.
\begin{prop}\label{prop:denseP}
Let $E$ be a Banach space, $p\in [1,\infty)$ and $k\geq 1$. Then $\mathscr{P}\otimes E$ is dense in $\D^{k,p}(E)$.
\end{prop}
\begin{proof}
It follows from \cite[Theorem 15.108]{Jan97} and \cite[Corollary 1.5.1]{Nualart2} that $\mathscr{P}$ is dense in $\D^{k,p}(\R)$.
Since $\D^{k,p}(\R)\otimes E$ contains the dense subspace $\mathscr{S}\otimes E$, the result follows from the definition of $\D^{k,p}(E)$.
\end{proof}

Below, the space $\DD^{1,p}(\gamma(K,E))$, where $K$ is an arbitrary Hilbert space, will play an important role for the divergence operator $\delta$.
The following result is a direct consequence of the  $\gamma$-Fubini lemma \ref{sec:gamma,Fubini-lemma}.
\begin{prop}\label{Malliavin-Fubini-isomorphism}
Let $K$ be a separable Hilbert space, $E$ be a Banach space and let $1\leq p<\infty$. Then the map ${\rm Fub}: \DD^{1,p}(\gamma(K,E)) \to \g(K,\DD^{1,p}(E))$ defined by \[(({\rm Fub}F)k)(\omega) := F(\omega)k, \qquad \omega\in \Omega, k\in H,\]
is an isomorphism $\DD^{1,p}(\gamma(H,E))\eqsim \gamma(H,\DD^{1,p}(E))$. Here
\[(D [({\rm Fub}F)k]) h = ((DF)h)k,  \  \ \ \ \  h\in H, \ k\in K. \]
\end{prop}
In particular, this result holds for the case $K  = H$.

\section{Results on Malliavin derivatives}

\subsection{Poincar\'e inequality and its consequences}

The following Poincar\'e inequality will be useful to us. A similar result for \umd spaces was obtained  \cite[Theorem 1]{NaSch} in the discrete setting using entirely different methods. We extend the scalar-valued proof from \cite[Proposition 1.5.8]{Nualart2}.

\begin{prop}[Poincar\'e inequality]\label{prop:poincare}
Let $E$ be a \umd space and let $p\in (1, \infty)$. For all $u\in \D^{1,p}(\O;E)$ one has
\begin{align*}
\|u-\EE(u)\|_{L^p(\Omega;E)} \lesssim_{p,E} \|Du\|_{L^p(\Omega;\gamma(H,E))},
\end{align*}
\end{prop}
\begin{proof}
Let $R$ be the operator on $L^p(\Omega;E)$ defined by $R = \sum_{n=1}^\infty \sqrt{1+1/n} J_n$. We will prove that $R$ is bounded, using Theorem \ref{Meyer's Multiplier Theorem}. With $\phi(n) := \sqrt{1+1/n}$, observe that
\[\phi(n) = \sum_{k=0}^\infty {1/2\choose k} n^{-k} = \sum_{k=0}^\infty \frac{(1/2)(1/2-1)\ldots (1/2-k+1)}{k!}n^{-k}.\]
With Stirling's formula we obtain
\[ \sum_{k=1}^\infty \Big|{1/2\choose k}\Big| = \sum_{k=1}^\infty \frac{(2k)!}{(2k-1)(k!)^24^k} \sim \sum_{k=1}^\infty \frac{1}{k^{3/2}}.\]
It follows that the above series converges, hence $R$ is bounded.

For any smooth $u\in \mathscr{S} \otimes E$ we have
\[ \|(I-L)^{-\frac{1}{2}}RCu\|_{L^p(\Omega;E)} = \Big\|\sum_{j=1}^\infty J_nu\Big\|_{L^p(\O;E)} = \|u-\EE(u)\|_{L^p(\O;E)}.\]
With approximation, it follows that the equality holds for all $u\in \DD^{1,p}(E)$. Using the boundedness of $R$, $(I-L)^{-1/2}$ and Meyer's inequalities, we obtain
\begin{align*}
\|u-\EE(u)\|_{L^p(\Omega;E)} &= \|(I-L)^{-1/2}RCu\|_{L^p(\O;E)} \leq c\|Cu\|_{L^p(\O;E)} \leq c'\|Du\|_{L^p(\Omega;\gamma(H,E))},
\end{align*}
\end{proof}

As a consequence of the Poincar\'e inequality one has the following:
\begin{clry}\label{cor:poincare}
Let $E$ be a \umd space, $p\in (1, \infty)$ and let $k\geq 1$. For all $u\in \D^{k,p}(\O;E)$ one has
\[\|u\|_{\D^{k,p}(E)} \eqsim_{p,E,k} \|u\|_{L^p(\O;E)} + \|D^k u\|_{L^p(\O;\g^k(H,E))}.\]
\end{clry}
\begin{proof}

By density it suffices to prove the norm equivalence for all $u\in \mathscr{S}\otimes E$.
The part $\gtrsim_{p,E,k}$ is trivial. For the estimate $\lesssim_{p,E}$, by an iteration argument it suffices to show that for all $i\geq 1$,
\begin{equation}\label{eq:suffiteration}
\|D^{i} u\|_{L^p(\O;\g^{i}(H,E))}\lesssim_{p,E,i} \|u\|_{L^p(\O;E)} + \|D^{i+1} u\|_{L^p(\O;\g^{i+1}(H,E))}.
\end{equation}
Observe that $\E D^{i} u = J_0 D^{i} u = D^{i} J_{i} u$ (see \cite[Proposition 1.2.2]{Nualart2}). Hence by \cite[Theorem 5.3]{janmaas} (applied $i-1$ times), one has
\[\|\E D^{i} u\|_{\g^i(H,E)} = \|D^{i} J_{i} u\|_{L^p(\O;\g^i(H,E))} \eqsim_{p,i} \|J_{i} u\|_{L^p(\O;E)}\lesssim_{p,E,i} \|u\|_{L^p(\O;E)}.\]
Now Proposition \ref{prop:poincare} and the latter estimate imply that
\begin{align*}
\|D^{i} u\|_{L^p(\O;\g^{i}(H,E))} & \leq \|\E D^{i} u\| + \|D^{i} u - \E D^{i} u\|_{L^p(\O;\g^{i}(H,E))}
\\ & \lesssim_{p,E,i} \|u\|_{L^p(\O;E)} + \|D^{i+1} u\|_{L^p(\O;\g^{i+1}(H,E))}.
\end{align*}
and hence \eqref{eq:suffiteration} follows.
\end{proof}

\subsection{Independence of $p$ and a weak characterization\label{subsec:indweak}}

The following theorem suggests that for $F$ to be in $\DD^{k,p}(E)$, it suffices to check that $F$ is differentiable in a very weak sense. The result is known in the case $E=\R$ (see \cite[Theorem 15.64]{Jan97}). However, in this situation the proof below is new as well.
\begin{thrm}\label{thm:pind and D-weak=strong-UMD}
Let $E$ be a \umd Banach space, let $p\in (1, \infty)$ and $k\geq 1$. Let $F\in L^p(\Omega;E)$ be such that for all $x^*\in E^*$ one has $\lb F, x^*\rb \in \D^{k,1}(\R)$. If there exists an $\xi \in L^p(\Omega;\gamma^k(H,E))$ such that for all $x^*\in E^*$
\[D^k\inprod{F, x^*} = \inprod{\xi, x^*}.\]
Then $F\in \DD^{k,p}(E)$ and $D^k F = \xi$.
\end{thrm}
\begin{proof}
Since $P_E$ is an analytic semigroup, $P_E(t) F \in \bigcap_{j\geq 1}\textrm{Dom}(L_E^j) \subset \DD^{\infty,p}(E)$ for all $t>0$. The inclusion follows from Meyer's inequalities (see Theorem \ref{Meyer's inequalities}).
From \cite[Lemma 6.2]{janmaas} we get for all $x^*\in E^*$,
\begin{align*}
\inprod{D^k P_E(t)F, x^*} &= D^k \inprod{P_E(t)F, x^*} = D^k P_{\RR}(t)\inprod{F, x^*} = e^{-k t}P_{\gamma^k(H,\RR)}D^k \inprod{F, x^*}
 \\ & = e^{-k t}P_{\gamma^k(H,\RR)}\inprod{\xi, x^*}
 = \inprod{e^{-k t}P_{\gamma^k(H,E)}\xi, x^*}.
\end{align*}
Hence $D^k P_E(t)F = e^{-k t}P_{\gamma^k(H,E)}(t)\xi$.

Now, let $t_n \downarrow 0$ as $n\to \infty$, and set $F_n = P_E(t_n)F$. Then, by the strong continuity of $(P(t))_{t\geq 0}$, we get
$F_n \to F$ in $L^p(\Omega;E)$ and $D^k F_n \to \xi$ in $L^p(\Omega;\gamma^k(H,E))$. From Corollary \ref{cor:poincare} it follows that $(F_n)_{n\geq 1}$ is a Cauchy sequence in $\DD^{k,p}(E)$. Hence by the closedness of $D$, we get $F\in \DD^{k,p}(E)$
and $D^k F = \xi$.
\end{proof}

\begin{rmrk}
A careful check of the above proof of Theorem \ref{thm:pind and D-weak=strong-UMD} shows that we can replace the assumption $\lb F, x^* \rb \in \D^{k,1}(\R)$ by by $\lb F, x^*\rb\in D^{1,1}(\R)$ and iteratively, for all $1\leq j\leq k-1$, $D^{j} \lb F, x^* \rb \in D^{1,1}(\g^{j-1}(H,E))$. As a consequence, for \umd spaces $E$, one obtains that our definition of $\D^{k,p}(E)$ coincides with the definition of \cite{janmaas}.
\end{rmrk}

Next we will give another definition of a Gaussian Sobolev space. For $p\in [1, \infty)$ and $k\in \NN$ let
\[\DD_*^{k,p}(E) :=\{F\in \DD^{k,1}(E): F\in L^p(\O;E), \ \forall \ 1\leq j\leq k \ \  D^{j} F\in L^p(\O;\g^j(H,E)) \}.  \]

The next result can be viewed as a Gaussian version of the Meyers-Serrin theorem for Sobolev spaces. For the scalar setting, a proof is provided in \cite[Theorem 15.64]{Jan97}. There, as in \cite{Bog10}, a different definition of $\D^{k,p}(\R)$ is given in terms of differentiability properties. Their definition coincides with our definition, since $\mathscr{P}$ is dense in both spaces (see Proposition \ref{prop:denseP}, and \cite[Theorem 15.108]{Jan97}).
\begin{clry}\label{cor:Dstar=D}
Let $E$ be a \umd Banach space, let $p\in [1,\infty)$ and $k\in \NN$. Then $\DD_*^{k,p}(E) = \DD^{k,p}(E)$.
\end{clry}

\begin{proof}
First note that the case $p=1$ is trivial. Let $p>1$.
Obviously, one has $\DD^{k,p}(E) \subseteq \DD_*^{k,p}(E)$. The converse result follows from Theorem \ref{thm:pind and D-weak=strong-UMD}.
\end{proof}

\subsection{Calculus results\label{sec:calculus}}

Next, we will prove several calculus results in the vector-valued setting. The following product rule will be useful later on.

\begin{lemma}[Product rule]\label{lem:productrule}
Let $E_0$, $E_1$ and $E_2$ be Banach spaces. Let $b:E_0\times E_1\to E_2$ be a bilinear operator with the property that there is a constant $C$ such that for all $x\in E_0$ and $y\in E_1$ one has $\|b(x, y)\|\leq C\|x\|\, \|y\|$.
Let $1 \leq p,q,r <\infty$ be such that $\frac{1}{p} + \frac{1}{q} = \frac{1}{r}$. If $F \in \DD^{1,p}(E_0)$ and $G\in \DD^{1,q}(E_1)$, then $b(F,G) \in \DD^{1,r}(E)$ and
\begin{equation}\label{eq:productruleid}
D [b(F,G)] = b(DF,G) + b(F,DG).
\end{equation}
Here $b(DF,G)h = b((DF)h,G)$ and $b(F,DG)h = b(F,(DG)h)$ for $h\in H$.
\end{lemma}
\begin{proof}
Using H\"older's inequality, one sees that for all $F \in L^{p}(\O;E_0)$ and $G\in L^q(\O;E_1)$, $b(F,G) \in L^{r}(\O;E_2)$ and
\begin{equation}\label{eq:Lrholder}
\|b(F,G)\|_{L^r(\O;E_2)}\leq C \|F\|_{L^{p}(\O;E_0)} \|G\|_{L^q(\O;E_1)}.
\end{equation}

If $F\in \mathscr{S} \otimes E_0$ and $G\in \mathscr{S} \otimes E_1$, \eqref{eq:productruleid} follows from a straightforward calculation and the product rule for ordinary derivatives. Moreover, observe that
\[  \|D[b(F,G)]\|_{L^r(\O;\g(H,E_2))} \leq \|b(DF,G)\|_{L^r(\O;\g(H,E_2))} + \|b(F,DG)\|_{L^r(\O;\g(H,E_2))}\]
Now by linearity it follows that pointwise in $\O$, we have
\begin{align*}
\|b(DF,G)\|_{\g(H,E_2)} & = \Big\|b\Big(\sum_{n\geq 1}\tilde{\g}_n DF h_n,G\Big)\Big\|_{L^2(\tilde{\O};E_2)}
\\ & \leq C\Big\|\sum_{n\geq 1}\tilde{\g}_n DF h_n\Big\|_{L^2(\tilde{\O};E_0)}  \|G\|_{E_1} = C\|DF\|_{\g(H,E_0)} \|G\|_{E_1}.
\end{align*}
Here $(\tilde{\g}_n)_{n\geq 1}$ is a sequence of standard Gaussian random variables on a probability space $(\tilde{\O},\tilde{\F}, \tilde{\P})$.

Similarly, one sees that $\|b(F,DG)\|_{\g(H,E_2)}\leq C\|F\|_{\g(H,E_0)} \|DG\|_{\g(H,E_1)}$ pointwise in $\O$. From H\"older's inequality we obtain
\begin{equation}\label{eq:eq:LrholderD}
\begin{aligned}
\|&D [b(F,G)]\|_{L^r(\O;\g(H,E_2))} \\ & \leq C\|DF\|_{L^p(\O;\g(H,E_0))} \|G\|_{L^q(\O;E_1)} +  C\|F\|_{L^p(\O;E_0)} \|DG\|_{L^q(\O;\g(H,E_1))}
\\ & \leq 2C\|F\|_{\D^{1,p}(\O;E_0)} \|G\|_{\D^{1,q}(\O;E_1)}.
\end{aligned}
\end{equation}

Now let $F \in \DD^{1,p}(E_0)$ and $G\in \DD^{1,q}(E_1)$. Choose sequences $(F_n)_{n\geq 1}$ and $(G_n)_{n\geq 1}$ of smooth random variables such that $\limn F_n = F$ in $\DD^{1,p}(E_0)$ and $\limn G_n = G$ in $\DD^{1,p}(E_1)$. Then by \eqref{eq:Lrholder}, $\limn b(F_n,G_n) = b(F,G)$ in $L^r(\O;E_2)$. Moreover, by \eqref{eq:eq:LrholderD} $(D b(F_n,G_n))_{n\geq 1}$ is a Cauchy sequence and hence convergent in $L^r(\O;\g(H,E_2))$. Since $D$ is closed, we obtain $b(F,G)\in \D^{1,r}(E_2)$. Furthermore, \eqref{eq:productruleid} follows from an approximation argument.
\end{proof}

Let $E$ be a Banach space. For a sequence $(x_n)_{n\geq 1}$ in $E$ and $x\in E$ we say that $\limn x_n = x$ weakly if for all $x^*\in E^*$ one has $\limn \lb x_n, x^*\rb =\lb x, x^*\rb$. Notation: $x_n\rightharpoonup x$. Recall that if $E$ is reflexive, then for every bounded sequence $(x_n)_{n\geq 1}$ in $E$ there is a subsequence $(n_k)_{k\geq 1}$ and an element $x\in E$ such that $x_{n_k}\rightharpoonup x$ as $k$ tends to infinity. Moreover, in this case $\displaystyle \|x\|\leq \liminf_{n\to \infty} \|x_n\|$.

\begin{lemma}[Compactness]\label{cr-tech-lemma}
Let $E$ be a reflexive Banach space and let $p\in (1, \infty)$. Let $(F_n)_{n\geq 1}$ be a sequence in $\DD^{1,p}(E)$ and $F\in L^p(\Omega;E)$.
Assume $F_n \rightharpoonup F$ in $L^p(\Omega;E)$ and that there is a constant $C$ such that for all $n\geq 1$, $\|DF_n\|_{L^p(\Omega;\gamma(H,E))}\leq C$. Then $F\in \DD^{1,p}(E)$ and $\|DF\|_{L^p(\O;\g(H,E))}\leq C$. Moreover, there exists a subsequence $(n_k)_{k\geq 1}$ such that $DF_{n_k} \rightharpoonup DF$.
\end{lemma}

\begin{proof}
Let $G = \{(\xi, D\xi): \xi\in \D^{1,p}(E)\}\subseteq L^p(\O;E)\times L^p(\O;\g(H,E))$. Since $D$ is a closed linear operator, $G$ is a closed linear subspace of $L^p(\O;E)\times L^p(\O;\g(H,E))$. As $E$ and $\g(H,E)$ are reflexive, the latter space is reflexive, and hence $G$ is reflexive as well.

As $F_n \rightharpoonup F$ in $L^p(\Omega;E)$, the uniform boundedness principle implies that $(F_n)_{n\geq 1}$ is bounded in $L^p(\Omega;E)$. Now this together with the assumptions yields that $(F_n, D F_n)_{n\geq 1}$ is a bounded sequence in $G$. Since $G$ is reflexive, it follows that there is a $(\zeta,D\zeta)\in G$ and a subsequence $(n_k)_{k\geq 1}$ such that $(F_{n_k}, D F_{n_k})\rightharpoonup (\zeta, D\zeta)$. Since $F_{n}\rightharpoonup F$, one has that $\zeta = F$, and hence $F\in \D^{1,p}(E)$ with $D F = D\zeta$. It follows that $D F_{n_k} \rightharpoonup DF$, and in particular, $\|DF\|_{L^p(\O;\g(H,E))}\leq C$.
\end{proof}

Next we extend the chain rule for the Malliavin derivative to the vector-valued setting.
\begin{prop}\label{ChainRule-2}
Let $E_0$ be a Banach space, let $E_1$ be a \umd Banach space and let $p\in (1, \infty)$. Suppose $\varphi: E_0\to E_1$ is Fr\'{e}chet differentiable and has a continuous and bounded derivative. If $F \in \DD^{1,p}(E_0)$, then $\varphi(F) \in \DD^{1,p}(E_1)$ with
\[D(\varphi(F)) = \varphi'(F)DF.\]
\end{prop}
\begin{proof}
Observe that for all $x,y\in E_0$, $\|\varphi(x)-\varphi(y)\|\leq C\|x-y\|$, where $C = \sup_{y\in E_0}\|\varphi'(y)\|_{\calL(E_0,E_1)}$. In particular,  for all $x\in E_0$, $\|\varphi(x)\|\leq C(1+\|\varphi(0)\|)$.

{\em Step $1$:} First assume  $E_1 = \RR$. \\
Suppose that $F$ is a smooth random variable of the form
\[F = \sum_{m=1}^M f_m(W(h_1), \ldots, W(h_n)) \otimes x_m.\]
Now consider the isomorphism $b:\textrm{sp}\{x_1, \ldots, x_M\} \to \RR^{M'}$, $M'\leq M$, which sends $x_i$ to $e_i$. Then obviously $\psi:\R^{M'}\to \textrm{sp}\{x_1, \ldots, x_M\}$ given by $\psi := \varphi \circ b^{-1}$ is Fr\'{e}chet differentiable and has a continuous and bounded derivative given by $\psi'(x)(y) = \varphi'(b^{-1}(x)) b^{-1}(y)$. Moreover, from the finite dimensional chain rule (see \cite[Proposition 1.2.3]{Nualart2}) we get $\varphi(F) = \psi(b(F))$ is in $\DD^{1,p}(\R)$ and $D(\varphi(F)) = \varphi'(F)DF$.

Now let $F\in \D^{1,p}(E_0)$. Choose sequence of smooth random variables $F_k$ converging to $F$ in $\DD^{1,p}(E_0)$. By going to a subsequence if necessary, we can assume that
$F_k\to F$ in $E_0$ and $D F_k\to DF$ in $\g(H,E_0)$ almost surely as $k\to \infty$.
Clearly, one has
\[\lim_{k\to \infty} \|\varphi(F_k) - \varphi(F)\|_{L^p(\O)} \leq C \lim_{k\to \infty}\|F_k-F\|_{L^p(\O)} = 0.\]
Moreover, by the above, $\varphi(F_k)\in \D^{1,p}(\O)$ for each $k\geq 1$ and
\begin{align*}
\|D&(\varphi(F_k)) - \varphi'(F)DF\|_{L^p(\Omega;H)} = \|\varphi'(F_k)DF_k - \varphi'(F)DF\|_{L^p(\Omega;H)} \\
 &\leq \|\varphi'(F_k)DF_k - \varphi'(F_k)DF\|_{L^p(\Omega;H)} + \|\varphi'(F_k)DF - \varphi'(F)DF\|_{L^p(\Omega;H)}
 \\ & \leq C\|DF_k - DF\|_{L^p(\Omega;H)} + \|\varphi'(F_k)DF - \varphi'(F)DF\|_{L^p(\Omega;H)}
\end{align*}
The first term clearly converges to zero. The second term converges to zero by the continuity and boundedness of $\varphi'$ and the Dominated Convergence Theorem. By the closedness of $D$ it follows that $\varphi(F) \in \DD^{1,p}(\R)$ and $D(\varphi(F)) = \varphi'(F)DF.$

{\em Step $2$:} Let $E_1$ be an arbitrary \umd space. Let $F\in \DD^{1,p}(E_0)$. Fix an $y^*\in E_1^*$. Consider the function $\Phi_{y^*}: E_0\to \RR$ defined by
\[\Phi_{y^*}(x) = \inprod{\varphi(x), y^*},  \ \ x\in E_0.\]
Applying step 1 to the function  $\Phi_{y^*}$ we obtain that
$\Phi_{y^*}(F) \in \DD^{1,p}(\RR)$ and
\[D\inprod{\varphi(F), y^*} = D(\Phi_{y^*}(F)) = \Phi'_{y^*}(F)DF = \inprod{\varphi'(F)DF, y^*}.\]
Since $y^*\in E_1^*$ was arbitrary, and $\varphi'(F)DF\in L^p(\O;\g(H,E_1))$, we can use Theorem \ref{thm:pind and D-weak=strong-UMD} to obtain that $\varphi(F) \in \DD^{1,p}(E_1)$ and $D(\varphi(F)) = \varphi'(F)DF$.
\end{proof}

\begin{rmrk}
It is clear from the proof of Proposition \ref{ChainRule-2} that it remains true if $E_1=\R$ and $p=1$.
\end{rmrk}

\subsection{A chain rule for Lipschitz functions\label{subsec:Lip}}
In this section we will study the chain rule for the Malliavin derivative for Lipschitz functions $\phi: E_0\to E_1$, where $E_0$ and $E_1$ are Banach spaces with some additional geometric structure.
\begin{prop}
Let $E_0$ be a Banach space that has a Schauder basis, let $E_1$ be a \umd Banach space and let $p\in (1, \infty)$. Let $\phi: E_0\to E_1$ be a Lipschitz function with
\[\|\phi(x_1)-\phi(x_2)\|\leq L \|x_1-x_2\|,  \ \ \ x_1, x_2\in E_0.\]
If $F\in \DD^{1,p}(E_0)$, then $\phi(F) \in \DD^{1,p}(E_1)$, and furthermore there exists a bounded linear operator $T_F:\gamma(H,E_0)\to L^\infty(\Omega;\gamma(H,E_1))$ with $\|T_F\|\leq L$ and $D(\phi(F)) = T_F(DF)$
\end{prop}
Observe that for $\xi\in L^p(\O;\g(H,E_0))$, $T_F(\xi)\in L^p(\O;\g(H,E_1))$ is well-defined as the composition of the mappings $\omega\mapsto (\omega, \xi(\omega))\in \O\times\g(H,E_0)$ and $(\omega, R)\mapsto (T_F R)(\omega)\in \g(H,E_1)$.

\begin{proof}
Let $(x_n)_{n\geq 1}$ be a Schauder basis for $E_0$ and let $(x_n^*)_{n\geq 1}$ be its associated biorthogonal functionals. We can assume that $\|x_n\| = 1$ for all $n\geq 1$. For each $n\geq 1$ consider the projection $S_n:E_0\to E_0$ onto the first $n$ basis coordinates. It is well-known that there is a constant $C$ such that for all $n\geq 1$, $\|S_n\| \leq C$. Letting, $S_0 = 0$, we see that $P_n := S_n - S_{n-1}$ satisfies $\|P_n\| \leq 2C$ for all $n\geq 1$.

For $n\geq 1$ fixed, consider the map $l_n: \RR^n \to \mathrm{sp}\{x_1, \ldots, x_n\}$ that sends the basis coordinate $e_i \in \RR^n$ to $x_i$.
We claim that $\|l_n\|\leq \sqrt{n}$ and $\|l_n^{-1}\|\leq 2C\sqrt{n}$.
Indeed, for $\alpha = (\alpha_1, \ldots, \alpha_n) \in \RR^n$ one has
\[\|l_n\alpha\| = \Big\|\sum_{i=1}^n \alpha_i x_i\Big\| \leq \sum_{i=1}^n |\alpha_i| \|x_i\| \leq \sqrt{n} \|\alpha\|,\]
and the first part of the claim follows.
For $x = \sum_{i=1}^n \alpha_i x_i \in \mathrm{sp}\{x_1, \ldots, x_n\}$ one has
\begin{align*}
2C\Big\|\sum_{i=1}^n \alpha_i x_i \Big\| \geq \Big\|P_j\sum_{i=1}^n \alpha_i x_i \Big\| = \|\alpha_j x_j\| = |\alpha_j|, \ \ \ j\in \{1,2,\ldots, n\}.
\end{align*}
It follows that
\[2\sqrt{n}C\Big\|\sum_{i=1}^n \alpha_i x_i \Big\| \geq \Big(\sum_{j=1}^n |\alpha_j|^2\Big)^{1/2} = \|l_n^{-1} x\|.\]
Hence the second part of the claim follows.

Next, for every $n\geq 1$, let $\zeta_n:\RR^n\to\RR$ be a $C^\infty(\R^n)$-function such that
\[\zeta_n\geq 0, \ \ \mathrm{supp}(\zeta_n) \subset B(0,1), \ \  \text{and} \ \ \int_{\RR^n} \zeta_n(x)\;dx = 1.\]
Fix $n\geq 1$ and fix $\e > 0$. Let $\zeta_n^{\e}:\R^n\to \R$ be given by by $\zeta_n^\e(x) := \e^{-n} \zeta_n(x/\e)$.
Define
$\phi_n: E_0\to E_1$ by
\[\phi_n(x) := \int_{\RR^n} \zeta_n^\e\big(y-l_n^{-1}(S_nx)\big) \phi(l_ny)\;dy = \int_{\RR^n} \zeta_n^\e(y)\phi\big(S_n(x)+l_n(y)\big)\;dy.\]
It follows that
\begin{align*}
(\EE\|\phi_n&(F) - \phi(F)\|_{E_1}^p)^{1/p} = \Big(\E\Big\|\int_{\RR^n} \zeta_n^\e(y)\big[\phi\big(S_n(F)+l_n(y)\big)  - \phi(F)\big]\;dy \Big\|_{E_1}^p\Big)^{1/p}
\\ & \leq \Big(\EE\Big(\int_{\RR^n} \zeta_n^\e(y)\|\phi\big(S_n(F)+l_n(y)\big) -\phi(F)\|_{E_1}\;dy\Big)^p\Big)^{1/p} \\
 &\leq L\Big(\EE\Big(\int_{\RR^n} \zeta_n^\e(y)\|S_n(F)+l_n(y)-F\|_{E_0}\;dy\Big)^p\Big)^{1/p} \\
 &\leq L\Big(\EE\Big(\int_{\RR^n} \zeta_n^\e(y)\|S_n(F)-F\|\;dy\Big)^p\Big)^{1/p} + L\int_{\RR^n}\zeta_n^{\e}(y) \|l_ny\|_{E_0}\;dy \\
 &\leq L(\EE\|S_n(F)-F\|^p)^{1/p} + L\sqrt{n}\int_{B(0, \e)} \e^{-n}\zeta_n(y/\e) \|y\|_{\R^n}\;dy \\
 &\leq L(\EE\|S_n(F)-F\|^p)^{1/p} + L\varepsilon \sqrt{n}.
\end{align*}
By the dominated convergence theorem one has
$(\EE\|S_n(F)-(F)\|^p)^{1/p} \to 0$ as $n\to\infty$. Therefore, letting $\e = \frac{1}{n}$.
it follows that $\limn \phi_n(F) = \phi(F)$ in $L^p(\Omega;E_1)$.

Clearly, $x\mapsto \zeta_n^{\e}(y-l_n^{-1}(S_nx))$ is Fr\'{e}chet differentiable, and hence $\phi_n$ is differentiable.
We claim that for all $x\in E_0$, $\|\phi_n'(x)\|\leq C L$. Indeed, fix $x,h\in E_0$ and note that
\[\phi_n'(x)h = \lim_{t\to 0} \frac{\phi_n(x+th) - \phi_n(x)}{t}. \]
Now for $t\neq 0$ one has
\begin{align*}
\Big\| \frac{\phi_n(x+th) - \phi_n(x)}{t}\Big\|_{E_1}
 &= \frac{1}{|t|} \int_{\RR^n} \zeta_n^{\e}(y) [ \phi(S_n(x)+S_n(th) + l_n(y))-\phi(S_n(x)+l_n(y))]\;dy\Big\|_{E_1} \\
 &\leq \frac{L}{|t|} \int_{\RR^n} \zeta_n^{\e}(y) \|tS_n(h)\|_{E_0}\;dy \leq L\|S_n h\|_{E_0} \leq LC \|h\|_{E_0}.
\end{align*}
Therefore, $\|\phi_n'(x)\|\leq LC$ and the claim follows.
By Proposition \ref{ChainRule-2}, we see that $\phi_n(F) \in \DD^{1,p}(E_1)$, with
\[D\phi_n(F) = \phi_n'(F)DF.\]
Moreover, by the above claim one obtains that
\[\|\phi'_n(F)DF\|_{L^p(\Omega;\gamma(H,E_1)} \leq LC\|DF\|_{L^p(\Omega;\gamma(H,E_0))}.\]
Since the latter is independent of $n$, we can use Lemma \ref{cr-tech-lemma} to conclude that $\phi(F) \in \DD^{1,p}(E_1)$. Moreover, taking an appropriate subsequence we can assume that
\begin{equation}\label{eq:weakconvLipschitz}
\limn \phi'_n(F)DF = D \phi(F) \ \ \ \  \text{in the weak topology of $L^p(\O;\g(H,E_1)$}.
\end{equation}
Next, we will show that there exists an operator $T \in \calL(\gamma(H,E_0),L^\infty(\Omega;\gamma(H,E_1)))$ such that $D(\phi(F)) = TDF$. Since $E_0$ has a basis, there exists a basis $(R_n)_{n\geq 1}$ for $\gamma(H,E_0)$. Set $T_n := \phi'_n(F)$. Replacing $(\O, \F, \P)$ by the space generated by $F$ and $DF$, we can assume $\O$ is countably generated. Moreover, since for each $n,j\geq 1$, $\phi_n'(F) R_j$ is strongly measurable, we can assume $E_1$ is separable and hence that $\g(H,E_1)$ is separable. Since $E_1$ is reflexive, it follows that $\g(H,E_1)$ is reflexive and hence also $\g(H,E_1)^*$ is separable. We can conclude that $L^1(\Omega;\gamma(H,E_1)^*)$ is separable. Moreover, once again by the reflexivity of $\g(H,E_1)$, one has $L^1(\Omega;\gamma(H,E_1)^*)^* = L^\infty(\O;\g(H,E_1))$.

Recall the following basic fact (see \cite[Theorem 3.17]{Ru91}): a bounded sequence $(x_n^*)_{n\geq 1}$ in $E^*$ where $E$ is a separable Banach space has a weak$^*$ convergent subsequence, i.e., there is an $x^*\in E^*$ such that for all $x\in E$, $\lb x, x^*\rb = \limk \lb x, x_{n_k}^*\rb$. Moreover, $\|x^*\|\leq \liminf_{k\to \infty} \|x_{n_k}^*\|$.

For every $\omega\in \O$, we can consider a canonical extension $T_n(\omega): \gamma(H,E_0) \to \gamma(H,E_1)$ defined by $(T_n(\omega) R) h = T_n(\omega) (R h)$, and this extension satisfies $\|T_n(\omega)\|\leq LC$. For every $R\in \gamma(H,E_0)$, the bounded sequence $(T_nR)_{n\geq 1}$ in $L^\infty(\Omega;\gamma(H,E_1))$ contains a weak$^*$ convergent subsequence. In particular, this holds for every element $R_i$ with $i\geq 1$. By a diagonal argument we can find a subsequence $(n_k)_{k\geq 1}$ and elements $(z_i)_{i\geq 1}$ in $L^\infty(\O;\g(H,E_1))$ such that for all $i\geq 1$, $\lim_{k\to \infty} T_{n_k} R_i = z_i$ in the weak$^*$-topology of $L^\infty(\Omega;\gamma(H,E_1))$.
Let $\g_0(H,E_0) = \textrm{sp}\{R_1, R_2, \ldots\}$. Define the operator $T:\g_0(H,E_0) \to L^\infty(\Omega;\gamma(H,E_1))$ by
\[T\Big(\sum_{i=1}^n a_iR_i\Big) = \sum_{i=1}^n a_iz_i.\]
For each $R\in \g_0(H,E)$ one has $\lim_{k\to \infty}T_{n_k} R = TR$ in the weak$^*$-topology of $L^\infty(\Omega;\gamma(H,E_1))$ and therefore,
\begin{align*}
\|TR\|_{L^\infty(\Omega;\gamma(H,E_1))} & \leq \liminf_{k\to\infty} \| T_{n_k} R\|_{L^\infty(\Omega;\gamma(H,E_1))} \leq LC\|R\|.
\end{align*}
It follows that $T$ has a continuous extension $T: \gamma(H,E_0)\to L^\infty(\Omega;\gamma(H,E_1))$. Moreover, an approximation argument shows that for all $R\in \g(H,E_0)$, $TR = \lim_{k\to \infty}T_{n_k} R$ in the weak$^*$-topology. We show that for all $\xi\in L^p(\O;\g(H,E_0))$ and all simple functions $\eta:\O\to \g(H,E_1)^*$ one has
\begin{equation}\label{eq:Txi identification}
\E \lb T\xi, \eta\rb = \lim_{k\to \infty} \E \lb T_{n_k} \xi, \eta\rb.
\end{equation}
To prove this, note that if $\xi$ is a simple function as well, then by linearity it suffices to prove \eqref{eq:Txi identification} for $\xi = \one_A R$ with $R\in \g(H,E_0)$ and $A\in \F$. In that case one has
\[\EE\langle \eta, T\xi\rangle = \EE\langle \one_A\eta, TR\rangle = \lim_{k\to\infty} \EE\langle \one_A \eta, T_{n_k} R\rangle= \lim_{k\to\infty} \EE\langle \eta, T_{n_k} \xi\rangle.
\]
for all $\eta\in L^1(\O;\g(H,E_1)^*)$.
Now let $\xi\in L^p(\O;\g(H,E_0))$ and let $\eta:\O\to \g(H,E)$ be simple function. Let $\varepsilon>0$ be arbitrary. Choose a simple function $\xi_0:\O\to \g(H,E_0)$ such that $\|\xi - \xi_0\|_{L^p(\O;\g(H,E_0))}\leq \varepsilon$. It follows that
\begin{align*}
\limsup_{k\to \infty}\big|\E \lb T\xi, \eta\rb - \E \lb T_{n_k} \xi, \eta\rb\big| & \leq \limsup_{k\to \infty} \big|\E\lb T \xi_0, \eta\rb  - \E \lb T_{n_k} \xi_0, \eta\rb\big| + 2LC \varepsilon
= 2LC \varepsilon.
\end{align*}
Since $\varepsilon>0$ was arbitrary \eqref{eq:Txi identification} follows.

Taking $\xi = DF$ in \eqref{eq:Txi identification} and using \eqref{eq:weakconvLipschitz} it follows that for all simple functions $\eta:\O\to \g(H,E_1)^*$ one has
\[\E \lb TDF, \eta\rb = \lim_{k\to \infty} \E \lb T_{n_k} DF, \eta\rb =  \lim_{k\to \infty} \E\lb \phi'_{n_k}(F)DF, \eta\rb = \E\lb D \phi(F), \eta\rb.\]
By a density and Hahn-Banach argument this yields $TDF = D\phi(F)$. Hence we can take $T = T_F$.
\end{proof}

\begin{rmrk}
The first part of the proof is based on the idea in \cite[Proposition 5.2]{CarmTeh}, where the result has been proved for Hilbert spaces $E_0$ and $E_1$. It is surprising that their argument can be extended to a Banach space setting. We do not know if the assumption that $E_0$ has a basis can be avoided. In the final part of the argument in \cite[Proposition 5.2]{CarmTeh} a compactness argument is used to construct an operator $T\in L^\infty(\Omega;\calL(E_0,E_1))$ such that $D(\phi(F)) = T(DF)$. There seems to be a difficulty in this proof, and at the moment it remains unclear whether such a $T$ exists. Note that there are subtle (measurability) differences between the latter space and $\calL(E_0,L^\infty(\Omega;E_1))$ if $E_0$ and $E_1$ are infinite dimensional.
\end{rmrk}

\section{The divergence operator and the Skorohod integral\label{sec:Skorohod}}

\begin{defn}
Let $p\in [1, \infty)$. Let ${\rm Dom}_{p,E}(\delta)$ be the set of $\zeta\in L^p(\Omega; \gamma(H,E))$ for which there exists an $F\in L^p(\Omega;E)$ such that
\[ \EE\inprod{\zeta, DG}_{\textrm{Tr}} = \EE\inprod{F, G}_{E, E^*}, \ \ \ \ \ G\in \mathscr{S}\otimes E^*.\]
In that case, $F$ is uniquely determined, and we write $\delta(\zeta) = F$. The operator $\delta$ with domain ${\rm Dom}_{p,E}(\delta)$ is called the {\em divergence operator}.
\end{defn}
The operator $\delta$ is closed and densely defined, which easily follows from the scalar setting (see \cite[p. 274]{Jan97} and \cite{Nualart2}). For $p\in (1, \infty)$, the operator $\delta$ coincides with the adjoint of $D$ acting on $\DD^{1,q}(E^*)$ where $\frac{1}{p} + \frac{1}{q} = 1$ (see \cite{janjan}).
If there is no danger of confusion, we will also write ${\rm Dom}(\delta)$ for ${\rm Dom}_{p,E}(\delta)$.

The following identity can be found in \cite[Lemma 3.2]{janjan}.
\begin{lemma}\label{sec:div-oprtr;lemma:simple-functions}
We have $\mathscr{S} \otimes \gamma(H,E) \subseteq \textrm{Dom}(\delta)$ and
\[\delta(f\otimes R) = \sum_{j\geq 1} W(h_j)f \otimes Rh_j - R(Df), \qquad f\in \mathscr{S},\ R\in \gamma(H,E).\]
Here, $(h_j)_{j\geq 1}$ denotes an arbitrary orthonormal basis of $H$.
\end{lemma}

An important consequence of Meyer's inequalities and multiplier theorem is the following sufficient condition to be in the domain of $\delta$ (see \cite[Proposition 6.10]{janmaas}).
\begin{prop}\label{delta-continuity-thrm}
Let $E$ be a \umd Banach space and let $1<p<\infty$. The divergence operator $\delta$ is continuous from $\DD^{1,p}(\gamma(H,E))$ to $L^p(\Omega;E)$.
\end{prop}

For Hilbert spaces $H_1$ and $H_2$, let us denote by $I_{H_1, H_2}$ the isomorphism
\begin{align}\label{eq:isomorphism I_H1H2}
I_{H_1, H_2}: \gamma(H_1, \gamma(H_2,E)) \to \gamma(H_2,\gamma(H_1,E)),
\end{align}
which is defined by $((I_{H_1, H_2} R)(h_2))(h_1)  = (R h_1)(h_2)$ for $h_1\in H_1$ and $h_2\in H_2$. We will write $I_{H_1} = I_{H_1, H_1}$.
The following proposition gives a certain commutation relation between $D$ and $\delta$.
\begin{prop}\label{commutation-relation}
Let $E$ be a \umd Banach space. If $u\in \DD^{2,p}(\gamma(H,E))$, then $\delta(u) \in \DD^{1,p}(E)$ and we have the relation
\[D(\delta(u)) = u + \delta(I_{H}(Du)).\]
\end{prop}
\begin{proof}
First, let $E=\RR$, and $u = f(W(h_1), \ldots, W(h_n))\otimes h$, with $h_1, \ldots, h_n\in H$ orthonormal and $h\in H$ such that $\|h\|=1$. We can use Lemma \ref{sec:div-oprtr;lemma:simple-functions} to obtain
\[D(\delta(u)) = \Big(\sum_{j=1}^n \partial_j f \otimes (h_j\otimes h)\Big)W(h) + f\otimes h -\sum_{j,k=1}^n \partial_{jk}f\otimes \lb h, h_j\rb h_k.\]
Another computation yields
\[ \delta(I_{H}(Du)) = \Big(\sum_{j=1}^n \partial_j f \otimes (h_j\otimes h)\Big)W(h) -\sum_{j,k=1}^n \partial_{jk}f\otimes \lb h, h_j\rb h_k.\]
The commutation relation can be extended by linearity. Now let $E$ be a \umd Banach space, $u\in \mathscr{S} \otimes \g(H,E)$. The commutation relation holds for $\lb u, x^*\rb$ for all $x^* \in E^*$, and hence it holds for $u$. For general $u\in \D^{2,p}(\g(H,E))$, the identity follows from Proposition \ref{delta-continuity-thrm} and an approximation argument.
\end{proof}
An immediate consequence is that Proposition \ref{delta-continuity-thrm} extends to $\D^{k,p}(\g(H,E))$ for $k\geq 1$.
\begin{clry}\label{sec:div-oprtr;clry:delta-cts-hghr-ordr}
Let $E$ be a \umd Banach space, $1<p<\infty$, and $k\geq 1$. The operator $\delta$ is continuous from $\D^{k,p}(\g(H,E))$ to $\D^{k-1, p}(E)$.
\end{clry}

Another consequence of \ref{delta-continuity-thrm} is that \cite[Proposition 1.5.8]{Nualart2} extends to the \umd-valued setting.
\begin{prop}\label{div-norm-estimate}
Let $E$ be a \umd space and let $1<p<\infty$. For all $u\in \DD^{1,p}(\gamma(H,E))$, we have
\[ \|\delta(u)\|_{L^p(\Omega;E)} \leq c_p(\|\EE(u)\|_{\gamma(H,E)} + \|Du\|_{L^p(\Omega;\gamma^2(H,E))}).\]
\end{prop}
\begin{proof}
Since $\delta$ is continuous from $\DD^{1,p}(\gamma(H,E))$ to $L^p(\Omega;E)$, we have
\[\|\delta(u)\|_{L^p(\Omega;E)} \leq c(\|u\|_{L^p(\Omega;\gamma(H,E))} + \|Du\|_{L^p(\Omega;\gamma^2(H,E))}).\]
By the triangle inequality, we have
\[\|u\|_{L^p(\Omega;\gamma(H,E))} \leq \|\EE(u)\|_{\gamma(H,E)} + \|u-\EE(u)\|_{L^p(\Omega;\gamma(H,E))}.\]
Now the result follows from Proposition \ref{prop:poincare}.
\end{proof}

\subsection{Independence of $p$ and weak characterization}
One can formulate the following analogue of Theorem \ref{thm:pind and D-weak=strong-UMD} for the divergence operator $\delta$.

\begin{thrm}\label{thm:pind and delta-weak=strong-UMD}
Let $E$ be a \umd Banach space, $p\in (1, \infty)$ and $k\geq 1$. Let $F\in L^p(\Omega;\g^k(H,E))$ be such that for all $x^*\in E^*$, $\lb F, x^*\rb$ is in $\textrm{Dom}_{1,\R}(\delta^k)$. If there exists a $\xi \in L^p(\Omega;E)$ such that for all $x^*\in E^*$ one has
\[\delta^k\inprod{F, x^*} = \inprod{\xi, x^*},\]
then $F\in \textrm{Dom}_{p,E}(\delta^k)$ and $\delta^k F = \xi$.
\end{thrm}

\begin{proof}
Since $E$ is a \umd Banach space, $\g^k(H,E)$ is as well, and as in the proof of Theorem \ref{thm:pind and D-weak=strong-UMD} we obtain that $P_{\g^k(H,E)}(t)$ is an analytic semigroup on $L^p(\Omega;\g^k(H,E))$ and
\[P_{\g^k(H,E)}(t) F \in \cap_{j\geq 1}\textrm{Dom}(L^j_{\g(H,E)}) \subset \DD^{k,p}(E)\subseteq \textrm{Dom}_{p,E}(\delta^k)\] for all $t>0$, where
the last inclusion follows from Corollary \ref{sec:div-oprtr;clry:delta-cts-hghr-ordr}.

By the symmetry of $(P(t))_{t\geq 0}$ and a duality argument, it follows from \cite[Lemma 6.2]{janmaas} that  $\delta^k (P_{\g^k(H,\R)}(t) G)  = e^{kt}P_{\RR}\delta^k G$ for all $G\in \textrm{Dom}_{1,\R}(\delta)$.
Hence for all $x^*\in E^*$,
\begin{align*}
\inprod{\delta^k (P_{\g^k(H,E)}(t)F), x^*} &= \delta^k \inprod{P_{\g^k(H,E)}(t) F, x^*} = \delta^k(P_{\g^k(H,\R)}(t)\inprod{F, x^*})\\& = e^{kt}P_{\RR}\delta^k (\inprod{F, x^*}) = e^{kt}P_{\RR}\inprod{\xi, x^*} = \inprod{e^{kt}P_{E}\xi, x^*}.
\end{align*}
Therefore, $\delta^k(P_{\g^k(H,E)}(t)F) = e^{kt}P_{E}(t)\xi$.

Now, let $t_n \downarrow 0$ as $n\to \infty$, and set $F_n = P_{\g^k(H,E)}(t_n)F$. Then, by the strong continuity of $(P(t))_{t\geq 0}$, we get
$F_n \to F$ in $L^p(\Omega;\g^k(H,E))$ and $\delta^k F_n \to \xi$ in $L^p(\Omega;E)$.
Hence, by closedness of $\delta^k$, we get $F\in \textrm{Dom}_{p,E}(\delta^k)$ and $\delta^k F = \xi$.
\end{proof}

\begin{rmrk}
If $H$ is replaced with $L^2(0,T;H)$ and $F:(0,T)\times\O\to \calL(H,E)$ is adapted, a weak characterization of the stochastic integral was given in \cite{NVW1} without assumptions on the filtration. Theorem \ref{thm:pind and delta-weak=strong-UMD} can be viewed as an extension to the non-adapted setting, but only under the additional assumption that the filtration is generated by $W$.
\end{rmrk}

\subsection{Additional results}

The next lemma is an integration by parts formula for the divergence operator.
\begin{lemma}[Integration by parts]\label{delta-intbyparts}
Let $E$ be a Banach space, and $p,q,r\in [1, \infty)$ such that $\frac{1}{p} + \frac{1}{q} = \frac{1}{r}$. Let $u\in L^p(\Omega;\gamma(H,E))$ and $F\in \DD^{1,q}(E^*)$. If $u\in \mathrm{Dom}(\delta)$, then $\inprod{u, F} \in \mathrm{Dom}_{r,\R}(\delta)$ and
\[\delta(\inprod{u, F}) = \inprod{\delta(u), F}_{E, E^*} - \inprod{u, DF}_{\mathrm{Tr}}.\]
\end{lemma}
\begin{proof}
Let $G\in \mathscr{S}$. Identifying $H$ with its dual, one obtains
\[ \inprod{DG, \inprod{u, F}}_H = \inprod{u, D G \otimes F)}_{\mathrm{Tr}}.\]
With Lemma \ref{lem:productrule}, we get
\begin{align*}
\EE\inprod{\inprod{u, F}, DG}_H & = \EE\inprod{u, DG\otimes F)}_{\mathrm{Tr}} \\
&= \EE\inprod{u, D(GF)}_{\mathrm{Tr}} - \EE\inprod{u, GDF}_{\mathrm{Tr}} \\
&= \EE\inprod{\delta(u), GF}_{E, E^*} - \EE\inprod{u, GDF}_{\mathrm{Tr}} \\
&= \EE (G\inprod{\delta(u), F}_{E, E^*}) - \EE (G\inprod{u, DF}_{\mathrm{Tr}}).
\end{align*}
Therefore, $\inprod{u,F}\in \mathrm{Dom}_{r,E}(\delta)$ and the identity follow.
Since $G$ was arbitrary, this yields the result by a density argument.
\end{proof}

The next lemma gives a relationship between the operators $D$, $\delta$ and $L$.
\begin{lemma}\label{sec:div-oprtr;rltion-D-delta-L}
Let $E$ be a \umd Banach space and $p\in (1,\infty)$. If $u\in \D^{2,p}(E)$, then $\delta (Du) = -Lu$.
\end{lemma}
\begin{proof}
Note that by Meyer's inequalities, we have $u\in \mathrm{Dom}(L)$. If $u\in \mathscr{S}\otimes E$, the claim follows from the scalar case (see \cite[Proposition 1.4.3]{Nualart2}). The general case follows from an approximation argument and Proposition \ref{delta-continuity-thrm} and Meyer's inequalities.
\end{proof}

\subsection{Preliminaries on the Skorohod integral}
In this section we recall the vector-valued It\^o integral and its extension to the non-adapted setting.

Assume $H = L^2(0,T;U)$ for some separable Hilbert space $U$, and some $T>0$.
The family $(W_U(t))_{t\in [0,T]}$ of mappings from $U$ to $L^2(\Omega)$ given by
\[W_U(t)u := W(\mathbf{1}_{[0,t]} \otimes u)\]
is a {\em $U$-cylindrical Brownian motion}.
For any $t\in [0,T]$, we denote by $\mathscr{F}_t$ the $\sigma$-algebra generated by $\{W_U(s)u:\ 0\leq s\leq t,\ u\in U\}$. Note that $\FF := (\mathscr{F}_t)_{t\in [0,T]}$ is a filtration. Let $\Phi: [0,T]\times \Omega \to \calL(U,E)$ be a {\em finite rank adapted step function}:
\[ \Phi(t, \omega) := \sum_{i=1}^m \sum_{j=1}^n \mathbf{1}_{(t_{i-1}, t_i]}(t) \mathbf{1}_{A_{ij}}(\omega) \sum_{k=1}^l u_k \otimes x_{ijk},\]
where $A_{ij}\in \mathscr{F}_{t_{i-1}}$ are disjoint for each $j$, and $(u_k)$ are orthonormal in $U$. For such processes, {\em the stochastic integral ${\rm Int}(\Phi)\in L^p(\O;E)$ with respect to $W_U$} is defined by
\[ {\rm Int}(\Phi) := \int_0^T \Phi(t)\; dW_U(t) := \sum_{i=1}^m\sum_{j=1}^n\sum_{k=1}^l \mathbf{1}_{A_{ij}} (W(t_i)u_k - W(t_{i-1})u_k) \otimes x_{ijk}.\]
Let $L^p_{\FF}(\Omega;\gamma(H,E))$ be the closure of the set of adapted finite rank step functions in $L^p(\Omega;\gamma(H,E))$.
Recall the following results (see \cite[Theorem 3.5]{NVW1} and \cite[Theorem 5.4]{janjan} respectively).
\begin{thrm}[Stochastic integral I]\label{thm:stochintI}
Let $E$ be a \umd Banach space and let $1<p<\infty$. The stochastic integral uniquely extends to a bounded operator
\[{\rm Int}: L^p_\FF(\Omega;\gamma(H,E)) \to L^p(\Omega;E).\]
In this case the process $(t,\omega)\mapsto {\rm Int}(\one_{[0,t]}\Phi)(\omega)$ has a continuous version and
for all $\Phi \in L^p_\FF(\Omega;\gamma(H,E))$ we have the two-sided estimate
\[\|{\rm Int}(\Phi)\|_{L^p(\Omega;C([0,T];E))} \eqsim_{p,E} \|\Phi\|_{L^p(\Omega;\gamma(H,E))}.\]
\end{thrm}
In the above result one does not need that $\FF$ is generated by $W$. If the above norm equivalence holds for all $\Phi\in L^p_\FF(\Omega;\gamma(H,E))$, then $E$ has the \umd property (see \cite{Ga1}).

\begin{thrm}[Stochastic integral II]\label{div-extenstion-thrm}
Let $E$ be a \umd space and $1<p<\infty$. The space $L^p_\FF(\Omega;\gamma(H,E))$ is contained in the domain of $\delta$ and for all $\Phi \in L^p_\FF(\Omega;\gamma(H,E))$ one has $\delta(\Phi) = {\rm Int}(\Phi)$.
\end{thrm}

Motivated by the above result, we will write
\[\int_0^T u(t)\; dW_U(t) = \delta(u),  \ \ \ u\in \textrm{Dom}(\delta),\]
and the latter is called the {\em Skorohod integral} of $u$.

\subsection{Stochastic integral processes\label{sec:stochintprocess}}

In this section we will assume that $H = L^2(0,T;U)$ for some Hilbert space $U$ and some $T>0$, and we will assume that $E$ is a \umd Banach space. With a slight abuse of notation, we will denote $\mathbf{1}_A: H\to H$, $A\in \mathcal{B}[0,T]$, as the bounded linear operator on $H = L^2(0,T;U)$ defined by
\[(\mathbf{1}_Ah)(t) := \mathbf{1}_A(t)h(t),\]
for almost every $t\in [0,T]$. \\
With another slight abuse of notation, we can in a similar way view $\mathbf{1}_A$ as an operator on $\gamma(H,E)$. Indeed, for $R\in \gamma(H,E)$ we define $(\mathbf{1}_AR)h := R(\mathbf{1}_Ah)$. From the ideal property yields that $\mathbf{1}_A$ is then indeed an operator on $\gamma(H,E)$.

When $u\in \textrm{Dom}(\delta)$, it does not generally hold that $\one_{(s,t]}u \in \textrm{Dom}(\delta)$. Indeed, already in the case $p=2$ and $E=\R$ a counterexample can be found in \cite{Nualart2} and \cite{Pronk}. Define
\[\LL^{p}(E) := \{u\in \textrm{Dom}_{p,E}(\delta):\ \mathbf{1}_{[0,t]}u \in \textrm{Dom}_{p,E}(\delta)\ \textrm{for all } t\in [0,T]\}.\]
For $u\in \LL^p(E)$ we define the process
\[\zeta(t) := \delta(\mathbf{1}_{[0,t]}u) :=\int_0^t u(s) \, dW_U(s),  \ \ \ t\in [0,T].\]
Note that $\DD^{1,p}(\g(H,E)) \subset \LL^p(E)$. Indeed, by Theorem \ref{Malliavin-Fubini-isomorphism}, one obtains that if $u\in \D^{1,p}(\g(H,E))$, then $\one_{[0,t]}u\in \D^{1,p}(\g(H,E))$. The inclusion then follows from Proposition \ref{delta-continuity-thrm}.

Below we will also need Banach spaces of type $2$. Let us recall the definition.
Let $p\in [1,2]$ and consider a Rademacher sequence $(r_n)$. The Banach space $E$ has {\em type $p$} if there is a constant $C_p$ such that for all finite sequences $x_1, \ldots, x_N\in E$,
\[\Big( \EE \Big\|\sum_{n=1}^N r_nx_n\Big\|^2 \Big)^{1/2} \leq C_p \Big( \sum_{n=1}^N \|x_n\|^p \Big)^{1/p}.\]
Elementary facts are that every Banach space has type $1$, every Hilbert space has type $2$. Also, if a Banach space has type $p$, then it has type $p_0$ for all $p_0\in [1,p]$. If $p\in [1, \infty)$, $E$ has type $p_0\in [1,2]$ and $(A, \mathscr{A}, \mu)$ is a measure space, then $L^{p}(A;E)$ has type $\min\{p,p_0\}$.

Recall from \cite[Theorem 11.6]{Neerven-Radon} that for type $2$ spaces $E$ one has the following embedding
\begin{equation}\label{eq:type2embedding}
L^2(0,T;\gamma(U, E))\hookrightarrow \g(H,E),
\end{equation}
where again $H = L^2(0,T;U)$. A consequence is that if $u\in L^2(0,T;\D^{1,p}(\g(U,E)))$, then $u\in \D^{1,p}(\g(H,E))$.
We will show that, under extra integrability conditions, $\zeta(t) := \int_0^t u(s) \, dW_U(s)$ has a continuous version.

\begin{thrm}\label{integral-process-continuous-version}
Let $E$ be a \umd Banach space with type $2$, let $2<p<\infty$ and suppose $u\in L^2(0,T;\DD^{1,p}(\gamma(U,E)))$. If $r\mapsto D(u(r))\in L^p([0,T];L^p(\Omega;\gamma(H,\gamma(U,E)))$, then the integral process $\zeta:[0,T]\times\O\to E$ defined by
\[\zeta(t) = \int_0^t u(s) \, dW_U(s),  \ \ \ t\in [0,T], \]
has a version with continuous paths.
\end{thrm}
\begin{proof}
Let $v\in \g(H,E)$ be given by $v = \E u$. By Theorem \ref{thm:stochintI} the process $Y(t) = \int_0^t v \, dW_U$ has a continuous version. Replacing $u$ by $u - \E u$, from now on we can assume $\|\EE(u)\|_{\gamma(H,E)} = 0$.
Proposition \ref{div-norm-estimate} yields
\begin{align*}
\EE\|\zeta(t) - \zeta(s)\|^p &= \E \|\delta(\mathbf{1}_{[s,t]}u)\|^p \\ & \lesssim_{p,E} \E \|D(\mathbf{1}_{[s,t]}u)\|^p_{\g(H,\g(H,E))} = \E \|\mathbf{1}_{[s,t]}(I_{H}(Du))\|^p_{\g(H,\g(H,E))}.
\end{align*}
Here, $I_{H}$ is the isomorphism given in (\ref{eq:isomorphism I_H1H2}). Since $E$ has type $2$, also $\gamma(H,E)$ has type $2$. Hence by \eqref{eq:type2embedding}
\[\|F\|_{L^p(\Omega;\gamma(H,\gamma(H,E)))} \lesssim_{E} \|F\|_{L^p(\Omega; L^2(0,T;\gamma(U, \gamma(H,E))))},\]
for all $F\in L^p(\Omega\times [0,T];\gamma(U, \gamma(H,E))))$. This yields, using H\"{o}lder's inequality,
\begin{align*}
\|\mathbf{1}_{[s,t]}(I_{H}(Du))\|^p_{\g(H,\g(H,E))} &\lesssim_{p,E} \EE\Big( \int_s^t \|I_{H,U}(D(u(r)))\|^2_{\gamma(U, \gamma(H,E))} \, dr\Big)^{\frac{p}{2}} \\
 &\leq |t-s|^{\frac{p}{2}-1} \EE\Big( \int_s^t \|I_{H,U}(D(u(r)))\|^p_{\gamma(U, \gamma(H,E))} \, dr\Big) \\
 &= |t-s|^{\frac{p}{2}-1} \int_s^t A(r) \, dr,
\end{align*}
where $A(r) := \E\|D(u(r))\|^p_{\gamma(H, \gamma(U,E))}$. By Fubini's theorem it follows that for all $\theta\in (0,1/2)$,
\begin{align*}
\EE\int_0^T &\int_0^T \frac{\|\zeta(t)-\zeta(s)\|^p_E}{|t-s|^{\theta p +1}} \; dt\, ds \lesssim_{p,E} \Big(\int_0^T \int_{s}^T  \frac1{|t-s|^{2-p(\frac{1}{2}-\theta)}}\int_s^t A(r) \, dr \, dt  \, ds \Big) \\
 &= \int_0^T \int_s^T \int_r^T \frac{A(r)}{|t-s|^{2-p(\frac{1}{2}-\theta)}}  \, dt \, dr \, ds \\
 &\lesssim_{p,\theta} \int_0^T \int_s^T \big(|r-s|^{p(\frac{1}{2}-\theta)-1} + |T-s|^{p(\frac{1}{2}-\theta)-1}\big) A(r) \, dr \, ds\\
 & \lesssim_{p,\theta} \int_0^T \big(r^{p(\frac{1}{2}-\theta)} + (T-r)^{p(\frac{1}{2}-\theta)} + T^{p(\frac{1}{2}-\theta)} \big) A(r) \,dr \\
 &\lesssim_{p,\theta} T^{p(\frac12 -\theta)}\int_0^T A(r) \,dr = \|D(u(r))\|^p_{L^p(\O\times [0,T];\gamma(H,\gamma(U,E)))} < \infty.
\end{align*}
Also observe that
\[\E \int_0^T \|\zeta(t)\|^p \, dt \lesssim_{p,E} T^{p(\frac12 -\theta)} \|D u\|^p_{L^p(\O\times [0,T];\gamma(H,\gamma(U,E)))}.\]

It follows that (see \cite[section 2]{Amann}) $\zeta\in L^p(\O;W^{\theta,p}(0,T;E))$.
If $\theta\in (1/p,1/2)$, it follows from the Sobolev embedding theorem (see \cite[Theorem $4.12$]{Adams}) that $\zeta\in L^p(\Omega;C^{0,\gamma}(0,T;E))$ for all $0<\lambda\leq \theta-\frac{1}{p}$ and
\begin{equation}\label{eq:Holderestproof}
\|\zeta\|_{L^p(\Omega;C^{0,\lambda}(0,T;E))} \lesssim_{E,p,\lambda,\theta,T} \|D u\|^p_{L^p(\O\times [0,T];\gamma(H,\gamma(U,E)))}.
\end{equation}
In particular, $\zeta$ has a continuous version.
\end{proof}

\begin{clry}
Assume the conditions of Theorem \ref{integral-process-continuous-version} hold. If additionally $\E u\in L^p(0,T;\g(U,E))$, then $\zeta$ has a version in $L^p(\O;C^{\lambda}([0,T];E))$ for all $\lambda\in (0,\frac12-\frac1p)$. Moreover, for every $\lambda\in (0,\frac12-\frac1p)$ there is a constant $C$ independent of $u$ such that
\[\|\zeta\|_{L^p(\O;C^{0,\lambda}([0,T];E))}\leq C\|\E u\|_{L^p(0,T;\g(U,E))} + C\|Du\|_{L^p(\O\times[0,T];\g(H,\g(U,E)))}.\]
\end{clry}
\begin{proof}
By the previous proof and in particular \eqref{eq:Holderestproof} it suffices to estimate the $L^p(\O;{C^{0,\lambda}([0,T];E)})$-norm of $\eta$, where $\eta:[0,T]\times\O\to E$ is given by $\eta(t) = \int_0^t v \, dW_U$ and $v = \E u$. It follows from Theorem \ref{thm:stochintI} and \eqref{eq:type2embedding} that for $0\leq s< t\leq T$ one has
\begin{align*}
(\E\|\eta(t) - \eta(s)\|^p)^{1/p} & \eqsim_{E,p} \|\one_{[s,t]}v\|_{\g(H,E)}
\\ & \leq \|\one_{[s,t]}v\|_{L^2(0,T;\g(U,E)}\leq |t-s|^{\frac12-\frac1p} \|v\|_{L^p(0,T;\g(U,E))}.
\end{align*}
Therefore, as in the proof of Theorem \ref{integral-process-continuous-version} one has that $\eta\in L^p(\O;W^{\theta,p}(0,T;E))$ for all $0<\theta<1/2$ and for all $\lambda\leq \theta-\frac1p$ one has
\[\|\eta\|_{L^p(\O;C^{0,\lambda}([0,T];E))}\lesssim_{p,\lambda,\theta,T} \|\eta\|_{L^p(\O;W^{\theta, p}(0,T;E))} \lesssim_{p,E} \|v\|_{L^p(0,T;\g(U,E))}.\]
\end{proof}

\section{It\^{o}'s formula in the non-adapted setting\label{sec:Ito1}}

In the setting of adapted processes with values in a \umd-Banach space $E$, a version of It\^o's formula has been obtained in \cite{NVW2}. A version for Banach spaces with martingale type $2$ was already obtained in \cite{Nei}.
Below in Theorem \ref{Ito-formula} we present a version for the Skorohod integral for UMD spaces with type $2$. For Hilbert spaces $E$ the result can be found in \cite{GrPar}. Our proof follows the arguments in the scalar-valued case of It\^o's formula from \cite[Theorem 3.2.2]{Nualart2}.

Consider the $E$-valued stochastic process given by
\begin{align}\label{sec:Ito,process-formula}
\zeta_t = \zeta_0 + \int_0^t v(s) \, ds + \int_0^t u(s) \, dW_U(s).
\end{align}
where $\zeta_0$, $u$ and $v$ are non-adapted, but satisfy certain smoothness assumptions.
We prove an It\^o formula for $F(\zeta)$, where $F:E\to \R$ is twice continuously Fr\'echet differentiable with bounded derivatives.

\subsection{Preliminary results for It\^{o}'s formula}
Next, we will prove a couple of lemmas that are used in It\^{o}'s formula. Let $U$ be a Hilbert space, and set $H = L^2(0,T;U)$.
\begin{lemma}\label{Ito-helplemma2}
Let $\xi \in L^2(\Omega;H)$ and $h\in H$. For $N\in \NN$, consider the partition $(t_n^N)_{n=0}^N$ of $(0,T)$, where $t_n^N = \frac{nT}{N}$. Then
\[\EE \sum_{n=1}^N \inprod{\one_{[t_{n-1}^N, t_n^N]}h, \xi}_H^2 \to 0, \qquad N\to\infty.\]
\end{lemma}
\begin{proof}
If $h = \one_{[t_{k-1}^K, t_k^K]} \otimes \varphi$ with $\varphi\in U$, the result can be checked using H\"older's inequality. By linearity it extends to linear combinations of such $h$. For general $h\in H$ one has that
\begin{align}\label{eq:genhform}
\EE\sum_{n=1}^N \inprod{\one_{[t_{n-1}^N, t_n^N]}h, \xi}_{H}^2 &\leq \EE\|\xi\|_H^2 \sum_{n=1}^N \|\one_{[t_{n-1}^N, t_n^N]}h\|_H^2 = \EE\|\xi\|^2 \|h\|^2_H.
\end{align}
Therefore, the case $h\in H$ can be proved by approximation and using \eqref{eq:genhform}.
\end{proof}

\begin{lemma}\label{Ito-helplemma4}
Let $(a,b)$ be an open interval, and consider a partition $(t_n^N)_{n=1}^N$ such that $t_{n+1}^N - t_n^N \to 0$ as $N\to \infty$ for all $n$. If $u_1, u_2 \in U$, then for all $p\in [1, \infty)$, $\lim_{N\to \infty} \xi_N  = (b-a)\inprod{u_1, u_2}$ in $L^2(\O)$, where
\[\xi_N = \sum_{n=1}^N ((W_U(t_{n+1}^N) - W_U(t_n^N))u_1)((W(t_{n+1}^N) - W_U(t_n^N))u_2),  \ \ \  N\geq 1.\]
\end{lemma}
\begin{proof}
Since $\E \xi_N = (b-a)\inprod{u_1, u_2}$, it suffices to shows that $\lim_{N\to \infty}\E\xi_N^2 = (b-a)^2|\inprod{u_1, u_2}|^2$. This follows from a straightforward computation.
\end{proof}

The next result will be presented and needed only for dyadic partitions, but actually holds for more general partitions.
\begin{thrm}\label{Ito-step1}
Let $U$ be a separable Hilbert space and $E$ a \umd-Banach space with type $2$. Set $t_i^n = \frac{iT}{2^n}$ for $n\geq 1$ and $i=0,1,\ldots,2^n$. For each $n\geq 1$ and $i=0, 1, \ldots, 2^n$, let $\sigma_i^n\in [t_i^n, t_{i+1}^n]$.
Let $Z, Z^1, Z^2, \ldots:[0,T] \times \Omega \to \calL(E,E^*)$ be processes and assume that
\begin{enumerate}[(i)]
\item All processes $Z, Z^1, Z^2, \ldots$ have continuous paths.
\item Pointwise on $\O$ one has $\displaystyle \limn \sup_{t\in [0,T]}\|Z(t)-Z^n(t)\|_{\calL(E,E^*)} = 0$.
\item There is a $C>0$ such that for all $t\in [0,T]$ and $\omega\in\O$, one has $\|Z^n(t,\omega)\|_{\calL(E,E^*)}\leq C$.
\end{enumerate}
Then for $u,v\in L^2(0,T;\D^{1,2}(\g(U,E)))$ one has
\begin{equation}\label{Ito-step1-formula}
\begin{aligned}
\sum_{i=0}^{2^n-1} \Big\langle &\int_{t_i^n}^{t_{i+1}^n} u(s)\;dW_U(s),  Z^n(\sigma_i^n)  \int_{t_i^n}^{t_{i+1}^n} v(s)\;dW_U(s)\Big\rangle
\\ & \to \int_0^T \langle u(s), Z(s)v(s)\rangle_{\mathrm{Tr}} \;ds \ \ \ \  \text{in $L^1(\O)$ as $n\to \infty$.}
\end{aligned}
\end{equation}
\end{thrm}

\begin{proof}
For notational convenience, let $X = L^2(0,T;\DD^{1,2}(\gamma(U,E)))$ in the proof below.
Fix $n\geq 1$.
Let $\Phi_n: X\times X\to L^1(\O)$ be given by
\begin{align*}
\Phi_n(u,v) = \sum_{i=0}^{2^n-1} \Big\langle &\int_{t_i^n}^{t_{i+1}^n} u(s)\;dW_U(s), Z^n(\sigma_i^n)  \int_{t_i^n}^{t_{i+1}^n} v(s)\;dW_U(s)\Big\rangle.
\end{align*}
Observe that the stochastic integrals are well-defined in $L^2(\O;E)$ by Proposition \ref{delta-continuity-thrm} and \eqref{eq:type2embedding}.
Let $\Phi: X\times X\to L^1(\O)$ be given by
\begin{align*}
\Phi(u,v) = \int_0^T \langle u(s), Z(s)v(s)\rangle_{\mathrm{Tr}} \;ds
\end{align*}
This is well-defined by Lemma \ref{lem:trace} and the remarks below it.
For proof of \eqref{Ito-step1-formula} it suffices to show that $\limn \Phi_n(u,v) = \Phi(u,v)$ in $L^r(\O)$. For this we proceed in four steps below.

\smallskip
{\em Step 1:} Uniform boundedness of the bilinear operator $\Phi_n$.

We first show that there exists an $M\geq 0$ such that for all $u, v\in X$ and for all $n\geq 1$ one has
\begin{equation}\label{eq:Phin ineq}
\|\Phi_n(u,v)\|_{L^1(\O)} \leq M\|u\|_{X} \|v\|_{X}.
\end{equation}

One has
\begin{align*}
\|\Phi_n(u,v)\| \leq C  \sum_{i=0}^{2^n-1} \Big\| \int_{t_i^n}^{t_{i+1}^n} u(s) \, dW_U(s) \Big\| \Big\| \int_{t_i^n}^{t_{i+1}^n} v(s) \, dW_U(s) \Big\|.
\end{align*}
Therefore, with $\|\cdot\|_1 = \|\cdot\|_{L^1(\O)}$,
\begin{align*}
\|\Phi_n(u,v)\|_1 &\leq C\Big(\EE \sum_{i=0}^{2^n-1} \Big\| \int_{t_i^n}^{t_{i+1}^n} u(s) \, dW_U(s) \Big\|^2\Big)^{\frac{1}{2}} \Big(\EE \sum_{i=0}^{2^n-1}\Big\|\int_{t_i^n}^{t_{i+1}^n} v(s) \, dW_U(s) \Big\|^2\Big)^{\frac{1}{2}}
\\ & \stackrel{(i)}{\lesssim_{E}} \sum_{i=0}^{2^n-1} \| \one_{[t_i^n,t_{i+1}^n]} u\|_{\D^{1,2}(\g(H,E))} \sum_{i=0}^{2^n-1} \| \one_{[t_i^n,t_{i+1}^n]} v\|_{\D^{1,2}(\g(H,E))}
\\ &\stackrel{(ii)}{\lesssim_E} \sum_{i=0}^{2^n-1} \| \one_{[t_i^n,t_{i+1}^n]} u\|_{X} \sum_{i=0}^{2^n-1} \| \one_{[t_i^n,t_{i+1}^n]} v\|_{X}
=
\|u\|_{X} \|v\|_{X}.
\end{align*}
Here (i) follows from Proposition \ref{delta-continuity-thrm}, and (ii) follows from \eqref{eq:type2embedding} and the fact that $E$ has type $2$.

\smallskip
{\em Step 2:} Boundedness of the bilinear operator $\Phi$.

As in Step 1, there exists an $M\geq 0$ such that for all $u, v\in X$ one has
\begin{equation}\label{eq:Phi ineq}
\|\Phi(u,v)\|_{L^1(\O)} \leq M\|u\|_{X} \|v\|_{X}.
\end{equation}
By Lemma \ref{lem:trace}, one has
\begin{align*}
|\langle u(s), Z(s)v(s)\rangle_{\mathrm{Tr}}| & \leq \|u(s)\|_{\g(U,E)} \|Z(s)v(s)\|_{\g(U,E^*)}
\leq C\|u(s)\|_{\g(U,E)} \|v(s)\|_{\g(U,E^*)}.
\end{align*}
Hence
\begin{align*}
\|\Phi(u,v)\|_{L^1(\O)}  & \leq \int_0^T C\|u(s)\|_{L^2(\O;\g(U,E))} \|v(s)\|_{L^2(\O;\g(U,E^*))} \, ds\leq C \|u\|_{X} \|v\|_{X}.
\end{align*}
Now \eqref{eq:Phi ineq} follows.

\smallskip
{\em Step 3}: Reduction to simple functions of smooth processes.

Let $(e_n)_{n=1}^\infty$ denote an orthonormal basis for $U$.
Note that the following functions form a dense subset of $X$.
\begin{equation}\label{eq:densevclass}
\sum_{j=0}^{2^p-1} \mathbf{1}_{[t_j^p, t_{j+1}^p]} \otimes g_j(W(e_1), \ldots, W(e_K)) \otimes (\sum_{l=1}^L \psi_{jl} \otimes y_{jl}),
\end{equation}
where the $g_j$'s are smooth, $\psi_{jl}\in U$ and $y_{jl}\in E$ for $1\leq l\leq L$ and $1\leq j\leq 2^{p}-1$.
Now assume \eqref{Ito-step1-formula} holds for all functions $u$ and $v$ of the form \eqref{eq:densevclass}. We will show that the general case with $u,v\in X$, follows from this by a continuity argument.

Let $u, v\in X$ be arbitrary.
Fix $\varepsilon\in (0,1)$. Define $\tilde{M}=\max\{\|u\|_{X}, \|v\|_{X}\}+1$.
Choose $\tilde{u}, \tilde{v}$ of the form \eqref{eq:densevclass} and such that
\[\|u-\tilde{u}\|_{X} < \varepsilon/(M\tilde{M}),  \ \ \  \|v-\tilde{v}\|_{X} < \varepsilon/(M\tilde{M}).\]
By \eqref{eq:Phin ineq} and using the bilinearity of $\Phi_n$ and writing $\|\cdot\|_1 = \|\cdot\|_{L^1(\O)}$ one obtains
\begin{align*}
\|\Phi_n(u,v) - \Phi_n(\tilde{u}, \tilde{v})\|_{1} & \leq \|\Phi_n(u,v - \tilde{v})\|_{1} + \|\Phi_n(u-\tilde{u}, \tilde{v})\|_{1}
\\ & \leq M\|u\|_{X} \|v-\tilde{v}\|_{X} +  M\|u-\tilde{u}\|_{X} \|\tilde{v}\|_{X}
 \leq 2\varepsilon.
\end{align*}
In a similar way one sees that $\|\Phi(u,v) - \Phi(\tilde{u}, \tilde{v})\|_{L^1(\O)} \leq 2\varepsilon$.
It follows that
\begin{align*}
\|&  \Phi_n(u,v) - \Phi(u,v)\|_{1} \\ & \leq \|\Phi_n(u,v) - \Phi_n(\tilde{u},\tilde{v})\|_{1} + \|\Phi_n(\tilde{u},\tilde{v}) - \Phi(\tilde{u},\tilde{v})\|_{1} + \|\Phi(\tilde{u},\tilde{v}) - \Phi_n(u,v)\|_{1}
\\ & \leq 4 \varepsilon + \|\Phi_n(\tilde{u},\tilde{v}) - \Phi(\tilde{u},\tilde{v})\|_{1}.
\end{align*}
Therefore, taking the $\limsup$ in the above estimate and using \eqref{Ito-step1-formula} for $\tilde{u}$ and $\tilde{v}$ one obtains that
\[\limsup_{n\to \infty} \|\Phi_n(u,v) - \Phi(u,v)\|_{1} \leq 4 \varepsilon.\]
Since $\varepsilon>0$ was arbitrary, it follows that $\limn \Phi_n(u,v) = \Phi(u,v)$ in $L^1(\O)$. This yields the result.

\smallskip
\textit{Step 4}: Convergence for simple functions of smooth processes.

We next prove \eqref{Ito-step1-formula} for $u$ and $v$ of the form \eqref{eq:densevclass}. By linearity it then suffices to consider
$u$ and $v$ of the form
\begin{align*}
u &= f(W(e_1), \ldots, W(e_K)) \otimes ((\mathbf{1}_{[a, b]} \otimes \varphi) \otimes x_1), \\
v &= g(W(e_1), \ldots, W(e_K)) \otimes ((\mathbf{1}_{[c, d]} \otimes \psi) \otimes x_2), \\
\end{align*}
for some $\varphi, \psi \in U$, $x_1, x_2 \in E$ and some dyadic intervals $[a,b]$ and $[c,d]$ and smooth $f,g:\R^K\to \R$. By homogeneity we can assume that $\|\varphi\|_U = \|\psi\|_U = 1$. Moreover, we can assume $a=c$ and $b=d$. Indeed, if $(a,b)\cap (c,d)= \varnothing$, then both sides  of \eqref{Ito-step1-formula} vanish for $n$ large enough. If $(a,b)\cap (c,d)\neq \varnothing$, then we can write $u$ and $v$ as a sum of smaller dyadic intervals which are either identical or disjoint. Furthermore, for notational convenience we assume that $[a,b] = [0,T]$.

Let $m\leq n$ and for $i=0,1,\ldots, n$ let us denote by $t_i^{(m)}$ the point of the $m$-th partition that is closest to $t_i^n$ from the left.
For each $n, m$ and $j$, let $S_{j}^{n,m} = \{i: t_i^n \in [t_j^m, t_{j+1}^m)\}$.
Then
\begin{align*}
&\Big| \sum_{i=0}^{2^n-1} \Big\langle \int_{t_i^n}^{t_{i+1}^n} u\; dW_U, Z^n(\sigma_i^n)\int_{t_i^n}^{t_{i+1}^n} v\; dW_U) \Big\rangle - \int_0^t \inprod{u(s), Z(s)v(s)}_{\mathrm{Tr}}\;ds \Big| \\
&\leq \Big| \sum_{i=0}^{2^n-1} \Big\langle \int_{t_i^n}^{t_{i+1}^n} u\; dW_U, (Z^n(\sigma_i^n)-Z(t_i^{(m)}))\int_{t_i^n}^{t_{i+1}^n} v\; dW_U \Big\rangle\Big| \\
&\; + \Big| \sum_{j=0}^{2^m-1} \sum_{i\in S_{j}^{n,m}} \Big\langle \int_{t_i^n}^{t_{i+1}^n} u\; dW_U, Z(t_j^m)\int_{t_i^n}^{t_{i+1}^n} v\; dW_U\Big\rangle - \int_{t_j^m}^{t_{j+1}^m} \inprod{u(s), Z(t_j^m)v(s)}_{\mathrm{Tr}}\;ds \Big| \\
&\; + \Big|\sum_{j=0}^{2^m-1} \int_{t_j^m}^{t_{j+1}^m} \inprod{u(s), Z(t_j^m)v(s)}_{\mathrm{Tr}} \;ds - \int_0^t \inprod{u(s), Z(s)v(s)}_{\mathrm{Tr}}\;ds\Big| \\
&= a_1+a_2+a_3.
\end{align*}

For the term $a_3$, pointwise in $\O$ one can estimate
\begin{align*}
a_3 &= \Big| \sum_{j=0}^{2^m-1} \int_{t_j^m}^{t_{j+1}^m} \inprod{u(s), (Z(t_j^m) - Z(s))v(s)}_{\mathrm{Tr}}\;ds\Big| \\
&\leq \sum_{j=0}^{2^m-1} \int_{t_j^m}^{t_{j+1}^m} \|u(s)\|_{\g(U,E)}\|v(s)\|_{\g(U,E)}\|Z(t_j^m) - Z(s)\|_{\calL(E,E^*)}\;ds \\
&\leq \sup_{|s-r|\leq T2^{-m}} \|Z(r)) - Z(s)\|_{\calL(E,E^*)} \, \|u\|_{L^2(0,T;\g(U,E))} \|v\|_{L^2(0,T;\g(U,E))},
\end{align*}
The latter converges to zero in $L^1(\O)$ as $n\to\infty$ and then $m\to\infty$ by the path-continuity of $Z$.

For $a_1$ pointwise in $\O$ one can estimate
\begin{align*}
a_1 &\leq \sup_{|s-r|\leq T2^{-m}} \|Z^n(r) - Z(s)\| \sum_{i=0}^{2^n-1} \|\delta(\one_{[t_i^n, t_{i+1}^n]}u)\|\|\delta(\one_{[t_i^n, t_{i+1}^n]}v)\| \\
 &\leq \sup_{|s-r|\leq T2^{-m}} \|Z^n(r) - Z(s)\| \Big(\sum_{i=0}^{2^n-1} \|\delta(\one_{[t_i^n, t_{i+1}^n]}u)\|^2+  \sum_{i=0}^{2^n-1} \|\delta(\one_{[t_i^n, t_{i+1}^n]}v)\|^2\Big).
\end{align*}
Define the uniformly bounded sequence of random variables $(z_{nm})$ and $(z_m)$ by
\[z_{nm} = \sup_{|s-r|\leq T2^{-m}} \|Z^n(r) - Z(s)\|,  \ \ \ \  z_m = \sup_{|s-r|\leq T2^{-m}} \|Z(r) - Z(s)\|.\]
Also let the random variables $(q_n(u))_{n\geq 1}$ and $q(u)$ be given by
\[q_n(u) = \sum_{i=0}^{2^n-1} \|\delta(\one_{[t_i^n, t_{i+1}^n]}u)\|^2,  \ \ \ \ \  q(u) = \int_0^t \|u(s)\|^2 \, ds.\]
We find
\begin{align*}
a_1 & \leq  z_{nm}(q_n(u) + q_n(v))
\\ & \leq z_{m} (q(u) + q(v)) + |z_{nm} - z_m|  (q(u)+q(v)) + z_{nm}\big( |q_n(u) - q(u)| + |q_n(v) - q(v)| \big)
\end{align*}
Since the ranges of $u$ and $v$ are one-dimensional in $E$ we can apply \cite[Theorem 3.2.1]{Nualart2} to obtain $\limn q_n(u) = q(u)$ in $L^1(\O)$ and similarly for $v$. Clearly, pointwise on $\O$, $\lim_{n\to \infty} z_{mn} = z_m$. Letting $n\to \infty$, the dominated convergence theorem gives that
\[\limsup_{n\to \infty} \E a_1 \leq \E (z_{m} (q(u) + q(v))).\]
Now letting $m\to \infty$ and again applying the dominated convergence theorem, we can conclude that $\lim_{m\to \infty} \lim_{n\to \infty} a_1 = 0$ in $L^1(\O)$.

We finish the proof once we have shown that $a_2\to 0$ in $L^1(\O)$.
For the moment, fix $j$. Let us calculate the second part of the summand of $a_2$:
\begin{align*}
\int_{t_j^m}^{t_{j+1}^m} \inprod{u(s), Z(t_j^m)v(s)}_{\mathrm{Tr}} \;ds
 &= \int_{t_j^m}^{t_{j+1}^m} \sum_{k=1}^\infty \inprod{f\otimes \inprod{h_1(s),e_k}x_1, Z(t_j^m)(g\otimes \inprod{h_2(s),e_k}x_2)}_{E, E^*}\;ds \\
&= fg\inprod{x_1, Z(t_j^m)x_2} \int_{t_j^m}^{t_{j+1}^m} \inprod{h_1(s), h_2(s)}\; ds,
\end{align*}
where $h_1 := \one_{[0,t]} \otimes \varphi$, $h_2 := \one_{[0,t]}\otimes\psi$. Let us compute the first part of $a_2$. For every $n\in \NN$, consider an orthonormal basis $(\tilde{h}_i)_{i=0}^\infty$ such that
\[\tilde{h}_i = \frac{1}{\sqrt{t_{i+1}^n - t_i^n}} \one_{[t_i^n, t_{i+1}^n]}\otimes \varphi,\ i=1,2,\ldots, 2^{n+1}.\]
Then by Lemma \ref{sec:div-oprtr;lemma:simple-functions} one has
\begin{align*}
\int_{t_i^n}^{t_{i+1}^n} u\;dW_U &= \Delta^n_i W \varphi f - \langle \one_{[t_i^n, t_{i+1}^n]}\varphi, Df \rangle_H)x_1.
\end{align*}
where $\Delta_i^nW = W_U(t_{i+1}^n) - W_U(t_i^n)$. A similar identity holds for the truncated Skorohod integral of $v$.
Therefore, one obtains
\begin{align*}
&\sum_{i\in S_{j}^{n,m}}\Big\langle \int_{t_i^n}^{t_{i+1}^n} u\; dW_U, Z(t_j^m) \int_{t_i^n}^{t_{i+1}^n} v\; dW_U\Big\rangle \\
&= \sum_{i\in S_{j}^{n,m}} (\Delta_i^nW\varphi f - \langle \one_{[t_i^n, t_{i+1}^n]}\varphi, Df \rangle_H) (\Delta_i^nW\psi g - \langle \one_{[t_i^n, t_{i+1}^n]}\psi, Dg \rangle_H) \inprod{x_1, Z(t_j^m)x_2}\\
&=: \sum_{i\in S_{j}^{n,m}} (A_i-B_i)(C_i-D_i)\inprod{x_1, Z(t_j^m)x_2}
\end{align*}
Thus the convergence would follow if
\begin{align*}
 \sum_{j=0}^{2^m-1}&\inprod{x_1, Z(t_j^m)x_2} \Big[\sum_{i\in S_{j}^{n,m}}(A_i-B_i)(C_i-D_i)  - fg\int_{t_j^m}^{t_{j+1}^m} \inprod{h_1(s), h_2(s)}\; ds\Big]\to 0
\end{align*}
in $L^1(\O)$ as $n\to \infty$ and then $m\to \infty$.
Pointwise on $\Omega$, the above expression is dominated by
\[C\|x_1\|\|x_2\|\Big| \sum_{i=0}^{2^n-1} (A_i-B_i)(C_i-D_i) - fg \int_0^t \inprod{h_1(s),h_2(s)}\;ds\Big|\]
Now it suffices to prove that
\begin{equation}\label{eq:suffABCD}
\EE\Big| \sum_{i=0}^{2^n-1} (A_i-B_i)(C_i-D_i) - fg \int_0^t \inprod{h_1(s),h_2(s)}\;ds\Big| \to 0,
\end{equation}
as $n\to\infty$.
To prove \eqref{eq:suffABCD} note that with $\|\cdot\|_1 = \|\cdot\|_{L^1(\O)}$ one has
\begin{align*}
\Big| &\sum_{i=0}^{2^n-1} (A_i-B_i)(C_i-D_i) - fg \int_0^t \inprod{h_1(s),h_2(s)}\;ds\Big\|_1 \\
&\leq \Big\|\sum_{i=0}^{2^n-1} A_iC_i - fg \int_0^t \inprod{h_1(s),h_2(s)}\;ds\Big\|_1 + \sum_{i=0}^{2^n-1} \|A_iD_i\|_1 + \|B_iC_i\|_1 + \|B_iD_i\|_1.
\end{align*}
We will show this $L^1(\Omega)$-convergence by showing the convergence for each of the components separately. First,
\begin{align*}
\EE \sum_{i=0}^{2^n-1} |B_iD_i| &\leq \Big(\EE \sum_{i=1}^{2^n-1} \inprod{\one_{[t_i^n, t_{i+1}^n]}h_1, Df}^2\Big)^{\frac{1}{2}} \Big(\EE \sum_{i=0}^{2^n-1} \inprod{\one_{[t_i^n, t_{i+1}^n]}h_2, Dg}^2\Big)^{\frac{1}{2}} \to 0,
\end{align*}
by Lemma \ref{Ito-helplemma2}. By the same lemma and the properties of $W$ one sees,
\begin{align*}
\sum_{i=0}^{2^n-1} \EE |A_i D_i| &= \sum_{i=0}^{2^n-1} \EE\Big|\Delta^n_i W \varphi f\cdot \inprod{\mathbf{1}_{[t_i^n, t_{i+1}^n]}h_2, Dg}_H\Big| \\
&\leq \|f\|_{\infty} \Big( \sum_{i=0}^{2^n-1} \EE ((\Delta_i^n W)\varphi)^2\Big)^{\frac{1}{2}} \Big(\sum_{i=0}^{2^n-1} \EE \inprod{\mathbf{1}_{[t_i^n, t_{i+1}^n]}h_2, Dg}_H^2\Big)^{\frac{1}{2}} \\
&\leq \|f\|_{\infty} \sqrt{T} \Big(\sum_{i=0}^{2^n-1} \EE \inprod{\mathbf{1}_{[t_i^n, t_{i+1}^n]}h_2, Dg}_H^2\Big)^{\frac{1}{2}} \to 0,
\end{align*}
and similarly $\EE \sum_{i=1}^{2^n-1} |B_iC_i| \to 0$. By Lemma \ref{Ito-helplemma4} one has
\begin{align*}
\EE &\Big|\sum_{i=0}^{2^n-1} A_iC_i - fg \int_0^t \inprod{h_1(s), h_2(s)}\;ds\Big| \\&\leq \|f\|_\infty\|g\|_\infty \Big( \EE\Big(\sum_{i=0}^{2^n-1} (\Delta_i^nW)(\varphi)(\Delta_i^nW)(\psi) - T\inprod{\varphi, \psi}\Big)^2\Big)^{\frac{1}{2}}\to 0
\end{align*}
as $n\to \infty$. Hence \eqref{eq:suffABCD} follows.
\end{proof}

Let $E$ be a \umd space with type $2$. Consider the following assumptions:
\begin{align}\label{eq:assumpLp}
\zeta_0 &\in \DD^{1,2}(E), & & D\zeta_0 \in L^2(\Omega;L^2(0,T;\gamma(U,E))) \nonumber \\
u &\in \DD^{2,2}(L^2(0,T;\gamma(U,E))), & & Du \in L^2(0,T;\DD^{1,2}(\gamma(U, \gamma(H,E)))), \\
v &\in \DD^{1,2}(L^2(0,T;E)), & & Dv \in L^1(0,T;L^2((0,T)\times\O;\gamma(U,E))). \nonumber
\end{align}
\begin{rmrk}\
\begin{enumerate}
\item[$(1)$] Clearly, $D\zeta_0$ is in $L^2(\O;\g(H,E))$ whenever $\zeta_0\in D^{1,2}(E)$. Note that the assumption \eqref{eq:assumpLp} states that $D\zeta_0$ is actually in the smaller space $L^2(\Omega;L^2(0,T;\gamma(U,E)))$. The latter space is smaller due to the type $2$ condition. The same applies to $Du$ and $Dv$.
\item[$(2)$] If $E$ is a Hilbert space, \eqref{eq:assumpLp} is equivalent to $\zeta_0 \in \D^{1,2}(E)$, $u\in \D^{2,2}(\g(H,E))$, and $v\in \D^{1,2}(L^2(0,T;E))$.
\end{enumerate}
\end{rmrk}
Let $\zeta:[0,T]\times\O\to E$ defined by
\begin{align}\label{eq:defzeta}
\zeta_t = \zeta_0 + \int_0^t v(r) dr + \int_0^t u(r) \, dW_U(r).
\end{align}
Observe for each $t\in [0,T]$, $\zeta(t)\in L^2(\O;E)$ is well-defined, by Proposition \ref{delta-continuity-thrm}. Moreover,
\begin{align*}
\sup_{t\in [0,T]} \|\zeta(t)\|_{L^2(\O;E)} & \lesssim_{E} \|\zeta_0\|_{L^2(\O;E)} + \|v\|_{L^1(0,T;L^2(\O;E))} + \|u\|_{\D^{1,2}(\g(H,E))}
\\ & \leq \|\zeta_0\|_{L^2(\O;E)} + \|v\|_{L^1(0,T;L^2(\O;E))} + \|u\|_{\D^{1,2}(L^2(0,T;\g(U,E)))},
\end{align*}
where we used \eqref{eq:type2embedding} in the last step.
In the next lemma we discuss differentiability properties of $\zeta$.
\begin{lemma}\label{lem:estDzeta}
Let $E$ be a \umd Banach space with type $2$. Assume that \eqref{eq:assumpLp} holds and let $\zeta$ be as in \eqref{eq:defzeta}. Set $Y = L^2(\O;L^2(0,T;\g(U,E)))$. Then for each $t\in [0,T]$, $\zeta(t)\in \D^{1,2}(E)$, $D(\zeta(t))\in Y$ and
\[(D\zeta(t))(s) = (D \zeta_0)(s) + \int_0^t (D v(r))(s) dr  + \one_{[0,t]}(s)u (s) + \int_0^t D(u(r))(s) \, dW_U(r),\]
\begin{align*}
\sup_{t\in [0,T]}\|&D\zeta(t)\|_{Y}  \lesssim_{p,E} \|D\zeta_0\|_{Y} + \|D v\|_{L^1(0,T;Y)} + \|u\|_Y \\ + & \|D u\|_{L^2(0,T;\D^{1,2}(\g(U,\g(H,E))))},
\end{align*}
\end{lemma}
\begin{proof}
The fact that $\zeta(t) \in \D^{1,2}(\O;E)$ and the first identity follow from Proposition \ref{commutation-relation}
The estimate follows from Proposition \ref{delta-continuity-thrm}.
\end{proof}

Define $D^-\zeta$ as the element in $Y=L^2(\Omega;L^2(0,T;\gamma(U,E)))$ given by
\begin{align}\label{sec:Ito,D^-formula}
(D^-\zeta)(s) = (D\zeta_0)(s) + \int_0^s (Dv(r))(s) dr + \delta(\mathbf{1}_{[0,s]}I_{U,H}((Du)(s))).
\end{align}
In the scalar case a more general definition of $D^-$ is given in \cite[p. 173]{Nualart2}. For processes of the form (\ref{eq:defzeta}), these definitions coincide (see \cite[Proposition 3.1.1]{Nualart2}). Observe that the last term in (\ref{sec:Ito,D^-formula}) can be written as $\int_0^s D(u(r))(s) \;dW_U(r)$. By our assumptions, this term is well-defined for almost all $s\in [0,T]$, and by continuity of $\delta$, we obtain
\begin{equation}\label{eq:stochintspeciaal}
\|s\mapsto \delta(\one_{[0,s]}I_{U,H}((Du)(s)))\|_Y\leq C\|Du\|_{L^2(0,T;\D^{1,2}(\g(H,\g(U,E)) ))}.
\end{equation}
Therefore, as in Lemma \ref{lem:estDzeta} one has
\begin{equation}\label{eq:estDminzeta}
\|D^{-}\zeta\|_{Y}  \lesssim_{p,E} \|D\zeta_0\|_{Y} + \|D v\|_{L^1(0,T;Y)} + \|D u\|_{L^2(0,T;\D^{1,2}(\g(H,\g(U,E))))}.
\end{equation}

\begin{lemma}\label{Ito-step4-formula}
Let $E$ be a \umd Banach space with type $2$.
Assume that \eqref{eq:assumpLp} holds and let $\zeta$ be as in \eqref{eq:defzeta}.
Suppose that $Z:\Omega\times[0,T] \to \calL(E,E^*)$ is bounded, and has continuous paths. Let $w \in L^2(0,T;L^2(\Omega;\gamma(U,E)))$ and $D^-\zeta$ as in (\ref{sec:Ito,D^-formula}). If we fix $t\in [0,T]$ and set $t_i^n := \frac{it}{2^n}$ for $i=0,1,\ldots, 2^n$, then
\[\sum_{i=0}^{2^n-1} \int_{t_i^n}^{t_{i+1}^n} \inprod{w(s), Z(t_i^n)D(\zeta(t_i^n))(s)}_{\mathrm{Tr}} \,ds \to \int_0^t \inprod{w(s), Z(s)((D^-\zeta)(s))}_{\mathrm{Tr}} \, ds,\]
in $L^1(\O)$, as $n\to\infty$.
\end{lemma}

\begin{proof}
Let
\[G(\zeta_0, v, u) = \|D\zeta_0\|_{Y} + \|D v\|_{L^1(0,T;Y)} + \|D u\|_{L^2(0,T;\D^{1,2}(\g(H,\g(U,E))))}\]
Let $\|Z\|_{\infty} = \sup_{s\in [0,T], \omega\in\O} \|Z(s,\omega)\|$, $\eta = D \zeta$ and $\xi = D^{-}\zeta$.
By Lemma \ref{lem:estDzeta}, $\inprod{w, Z(t_i^n)(\eta(t_i^n))(s)}_{\mathrm{Tr}}$ is well-defined a.e. in $(0,T)\times\O$. Moreover, one has
\begin{align*}
\E \sum_{i=0}^{2^n-1}  \int_{t_i^n}^{t_{i+1}^n} |\inprod{w, Z(t_i^n)(\eta(t_i^n))(s)}_{\mathrm{Tr}}| \, ds
& \leq \sum_{i=0}^{2^n-1}  \int_{t_i^n}^{t_{i+1}^n} \E  \|w(s)\|_{\g(U,E)} \|Z(t_i^n)(\eta(t_i^n))(s)\|_{\g(U,E^*)}\, ds
\\ & \leq C \|Z\|_\infty \|w\|_{Y} (G(\zeta_0, v, u) + \|u\|_Y).
\end{align*}
Similarly, by \eqref{eq:estDminzeta}, $\inprod{w(s), Z(s)((D^-\zeta)(s))}_{\mathrm{Tr}}$ is well-defined a.e.\ and
\begin{align*}
\E\int_0^t & |\inprod{w(s), Z(s)((D^-\zeta)(s))}_{\mathrm{Tr}}| \, ds \lesssim_{p,E} \|Z\|_\infty \|w\|_{Y} \, G(\zeta_0, v, u).
\end{align*}

Now observe that
\begin{align*}
\E\Big|&\sum_{i=0}^{2^n-1} \int_{t_i^n}^{t_{i+1}^n} \inprod{w(s), Z(t_i^n)(D(\zeta(t_i^n))(s))}_{\mathrm{Tr}}\, ds - \int_0^t \inprod{w(s), Z(s)((D^-\zeta)(s))}_{\mathrm{Tr}} \, ds\Big| \\
& \leq
\E\Big|\sum_{i=0}^{2^n-1} \int_{t_i^n}^{t_{i+1}^n} \inprod{w(s), Z(t_i^n)((D(\zeta(t_i^n)))(s) -(D^{-}\zeta)(s))}_{\mathrm{Tr}}\, ds \Big| \\
& \qquad+
\E\Big|\sum_{i=0}^{2^n-1} \int_{t_i^n}^{t_{i+1}^n} \inprod{w(s), (Z(t_i^n) -Z(s))(D^{-}\zeta)(s)}_{\mathrm{Tr}}\, ds \Big|
= T_{1, n} + T_{2,n}
\end{align*}
Let $\theta_n = \sup_{|r-s|<T2^{-n}} \|Z(r) -Z(s)\|_{\calL(E, E^*)}$. Note that pointwise in $\O$, $\limn \theta_n =0$ and $\theta_n\leq 2\|Z\|_{\infty}$. For each $n\geq 1$ one has
\begin{align*}
T_{2,n} &\leq \E\Big( \theta_n \int_{0}^{t} \|w(s)\|_{\g(U,E)} \, \|(D^{-}\zeta)(s)\|_{\g(U,E)} \, ds \Big)
\\ & \leq \E ( \theta_n  \|w\|_{L^2(0,T;\g(U,E))} \|D^{-}\zeta\|_{L^2(0,T;\g(U,E))} ).
\end{align*}
Since $\|w\|_{L^2(0,T;\g(H,U))} \|D^{-}\zeta\|_{L^2(0,T;\g(H,U))}\in L^1(\O)$, it follows from the dominated convergence theorem and the properties of $(\theta_n)_{n\geq 1}$, that $\limn T_{2,n} = 0$.

It follows from Lemma \ref{lem:estDzeta} that
\begin{align*}
T_{1,n} & \leq \|Z\|_\infty \E \sum_{i=0}^{2^{n}-1} \int_{t_i^n}^{t_{i+1}^n} \|w(s)\|_{\g(U,E)} \| D(\zeta(t_i^n))(s) -(D^{-}\zeta)(s)\|_{\g(U,E)} \, ds
\\ & \leq \|Z\|_\infty  \|w\|_{Y} \Big(\sum_{i=0}^{2^{n}-1} \int_{t_i^n}^{t_{i+1}^n} \E\| D(\zeta(t_i^n))(s) -(D^{-}\zeta)(s)\|_{\g(U,E)}^2 \, ds\Big)^{1/2}
\\ & {\leq} \|Z\|_\infty  \|w\|_{Y}  \Big(\sum_{i=0}^{2^{n}-1} \int_{t_i^n}^{t_{i+1}^n}  \E\Big|\int_{t_i^n}^s \|(Dv(r))(s)\|_{{\g(U,E)}} dr\Big|^2\, ds\Big)^{1/2} \\ & \qquad + \|Z\|_\infty  \|w\|_{Y}  \Big(\sum_{i=0}^{2^{n}-1} \int_{t_i^n}^{t_{i+1}^n}  \E\|\delta(\mathbf{1}_{[t_{i}^n,s]}(Du)(s))\|_{\g(U,E)}^2 \big)^{1/2} = S_{1,n} + S_{2,n}.
\end{align*}
For $S_{1,n}$ one has
\begin{align*}
S_{1,n} \leq \|Z\|_\infty  \|w\|_{Y}  \int_0^t \Big(\Big|\int_0^t  \one_{|r-s|<T2^{-n}} \E \|(Dv(r))(s)\|_{{\g(U,E)}}^2 \, ds\Big)^{1/2} \, dr.
\end{align*}
The latter converges to zero by the dominated convergence theorem and the assumption on $v$.
For $S_{2,n}$ one has
\begin{align*}
S_{2,n} & \leq \|Z\|_\infty  \|w\|_{Y}  \Big(\sum_{i=0}^{2^n-1} \int_{t_{i}^n}^{t_{i+1}^n} \|\one_{B(s,T2^{-n})} (Du)(s)\|_{\D^{1,2}(\g(H,\g(U,E)))}^2 \, ds\Big)^{1/2}
\\ & = \|Z\|_\infty  \|w\|_{Y}  \Big(\int_{0}^{t} \|\one_{B(s,T 2^{-n})} (Du)(s)\|_{\D^{1,2}(\g(H,\g(U,E)))}^2 \, ds\Big)^{1/2}.
\end{align*}
The latter converges to zero by the dominated convergence theorem and the assumption on $u$.
\end{proof}

\subsection{Formulation and proof of It\^{o}'s formula}

\begin{thrm}[It\^{o}'s formula]\label{Ito-formula}
Let $E$ be a \umd Banach space with type $2$. Suppose that the conditions \eqref{eq:assumpLp} hold and let $\zeta:[0,T]\times\O\to E$ be as in \eqref{eq:defzeta}. Assume $\zeta$ has continuous paths.
Let $F:E\to \RR$ be a twice continuously Fr\'{e}chet differentiable function. Suppose that $F'$ and $F''$ are bounded. Then
\begin{align*}
F(\zeta_t) &= F(\zeta_0) + \int_0^t F'(\zeta_s)(v(s)) \, ds + \delta(\inprod{F'(\zeta),\mathbf{1}_{[0,t]}u})\\&\ +\frac{1}{2}\int_0^t \inprod{u(s), F''(\zeta_s)(u(s))}_{\mathrm{Tr}} \, ds + \int_0^t \inprod{u(s), F''(\zeta_s)((D^-\zeta)(s))}_{\mathrm{Tr}} \, ds.
\end{align*}
\end{thrm}

Note that the term with $D^{-}$ is an additional term which is not present in the adapted setting.
A similar result in the case that $E$ is a Hilbert space can be found in \cite{GrPar}. Our proof is based on the ideas in \cite[Theorem 3.2.2]{Nualart2}.

\begin{rmrk}\
\begin{enumerate}
\item If $F'$ and $F''$ are not bounded, one can usually approximate $F$ with a sequence of functions that does satisfy the smoothness and boundedness conditions. In particular, such a procedure works in the important case where $F:E\to \R$, where $F(x) = \|x\|^s$, $s\geq 2$ and $E$ is an $L^q$-space with $q\geq 2$.
\item If the condition \eqref{eq:assumpLp} is strengthened to
\begin{align*}
\zeta_0 &\in \DD^{1,p}(E), & & D\zeta_0 \in L^p(\Omega;L^2(0,T;\gamma(U,E))) \nonumber \\
u &\in \DD^{1,p}(L^2(0,T;\gamma(U,E))), & & Du \in L^p(0,T;\DD^{1,p}(\gamma(H, \gamma(U,E)))), \\
v &\in \DD^{1,p}(L^2(0,T;E)), & & Dv \in L^1(0,T;L^p(\Omega;L^2(0,T;\gamma(U,E)))), \nonumber
\end{align*}
for some $p>2$, then by Lemma \ref{lem:productrule} and Proposition \ref{ChainRule-2} one actually has $\inprod{F'(\zeta),\mathbf{1}_{[0,t]}u})\in \D^{1,p/2}(H)$.
\item Using Theorem \ref{thm:pind and delta-weak=strong-UMD} in the same way as in Proposition \ref{ChainRule-2} one could extend the result to functions $F:E\to E_1$, where $E_1$ another \umd Banach space. However, in that case the traces have to be extended to the vector-valued setting as well.
\item Sufficient conditions for the existence of a continuous version of $\zeta$ can be found in Theorem \ref{integral-process-continuous-version}.
\end{enumerate}
\end{rmrk}

\begin{proof}
Set $t^n_i = \frac{it}{2^n}$, $0\leq i\leq 2^n$. Consider the Taylor expansion of $F(\zeta_t)$ up to the second order
\begin{align*}
F(\zeta_t) = \ &F(\zeta_0) + \sum_{i=1}^{2^n-1} F'(\zeta(t_i^n))(\zeta(t_{i+1}^n) - \zeta(t_i^n)) \\
 &+ \sum_{i=0}^{2^n-1} \frac{1}{2} \inprod{ \zeta(t^n_{i+1}) - \zeta(t^n_i), F''(\overline{\zeta}_i^n)(\zeta(t^n_{i+1}) - \zeta(t^n_i))}_{E, E^*}.
\end{align*}
Here, $\overline{\zeta}_i^n$ denotes a random intermediate point on the line between $\zeta(t_i^n)$ and $\zeta(t_{i+1}^n)$. It is well known that this can be done in such a way that $\overline{\zeta}_i^n$ is measurable. Now the proof will be decomposed in several steps.

\smallskip
\textit{Step 1:} We show that
\[\sum_{i=0}^{2^n-1} \inprod{ \zeta(t^n_{i+1}) - \zeta(t^n_i), F''(\overline{\zeta}_i^n)(\zeta(t^n_{i+1}) - \zeta(t^n_i))}_{E, E^*} \to \int_0^t \lb u(s), F''(\zeta_s)u_s\rb_{\mathrm{Tr}} \, ds\]
in $L^1(\O)$. Note that the increment $\zeta(t^n_{i+1}) - \zeta(t^n_i)$ equals
\[ \zeta(t^n_{i+1}) - \zeta(t^n_i) = \int_{t_i^n}^{t_{i+1}^n}  v(s) \, ds + \int_{t_i^n}^{t_{i+1}^n} u(s)\, dW_U(s). \]
Therefore, we can divide \[\sum_{i=0}^{2^n-1}\inprod{ \zeta(t^n_{i+1}) - \zeta(t^n_i), F''(\overline{\zeta}_i^n)(\zeta(t^n_{i+1}) - \zeta(t^n_i))}_{E, E^*}\] into $4$ parts. Consider the first piece
\[ \sum_{i=0}^{2^n-1} \Big\langle \int_{t_i^n}^{t_{i+1}^n}  v(s) \, ds, F''(\overline{\zeta}_i^n)\Big( \int_{t_i^n}^{t_{i+1}^n}  v(s) \, ds \Big) \Big\rangle.\]
Pointwise in $\O$ and for all $i, n$, one has
\begin{align*}
\Big|\sum_{i=0}^{2^n-1} \Big\langle \int_{t_i^n}^{t_{i+1}^n}  v(s) \, ds, F''(\overline{\zeta}_i^n)\Big( \int_{t_i^n}^{t_{i+1}^n}  v(s) \, ds \Big) \Big\rangle\Big| & \leq \|F''\|_{\infty} \sum_{i=0}^{2^n-1}\Big\| \int_{t_i^n}^{t_{i+1}^n}  v(s) \, ds \Big \|^2 \\
&\leq \|F''\|_{\infty} T2^{-n} \|v\|^2_{L^2(0,T;E)}
\end{align*}
The latter clearly goes to zero in $L^1(\O)$ as $n\to \infty$.
Next, both the second and the third part are pointwise dominated by
\begin{align*}
\|F''\|_{\infty} \sum_{i=0}^{2^n-1} \Big\| \int_{t_i^n}^{t_{i+1}^n} v(s)\; ds \Big\| \Big\| \int_{t_i^n}^{t_{i+1}^n} u(s)\; dW_U(s)\Big\| =: \xi_n
\end{align*}
We show that $\limn \xi_n = 0$ in $L^1(\O)$. Indeed, by Meyer's inequalities one has
\begin{align*}
\EE\sum_{i=0}^{2^n-1} &\Big\| \int_{t_i^n}^{t_{i+1}^n} v(s)\; ds \Big\| \Big\| \int_{t_i^n}^{t_{i+1}^n} u(s)\; dW_U(s)\Big\| \\
 &\leq \sqrt{t2^{-n}} \|v\|_{L^2(0,T\times\O;E)} \Big( \sum_{i=0}^{2^n-1} \|\mathbf{1}_{[t_i^n, t_{i+1}^n]} u\|^2_{\DD^{1,2}(\gamma(H,E))}\Big)^{\frac{1}{2}}.
\end{align*}
Now by the type $2$ assumption we have
\begin{align*}
\sum_{i=0}^{2^n-1} \|\mathbf{1}_{[t_i^n, t_{i+1}^n]} u\|^2_{\DD^{1,2}(\gamma(H,E))}& \lesssim_E \sum_{i=0}^{2^n-1} \|\mathbf{1}_{[t_i^n, t_{i+1}^n]} u\|^2_{\DD^{1,2}(L^2(0,T;\gamma(U,E)))}
\leq \|u\|^2_{\DD^{1,2}(L^2(0,T;\gamma(U,E)))}.
\end{align*}
Therefore, we find that $\limn \xi_n = 0$ in $L^1(\O)$, from which we see that the second and third term converge to zero in $L^1(\O)$.

To finish step 1, observe that $Z = F''\circ \zeta$ that continuous paths. Moreover, the process $Z^n=F''\circ \zeta_n$ where $\zeta_n$ is the process obtained by letting $\zeta_n(t_i^n) = \overline{\zeta}_i^n$, $i=0, \ldots, 2^n$, and by linear interpolation at the intermediate points. Then by the pathwise continuity of $\zeta$, it is clear that pointwise in $\O$, $\limn \sup_{t\in [0,T]}\|Z^n(t) - Z(t)\| =0$.
Hence, by Theorem \ref{Ito-step1} with $\sigma_i^n = t_i^n$,
\[\sum_{i=0}^{2^n-1} \Big\langle \int_{t_i^n}^{t_{i+1}^n} u(s) \; dW_U(s), F''(\overline{\zeta}_i^n) \int_{t_i^n}^{t_{i+1}^n} u(s) \; dW_U(s) \Big\rangle \to \int_0^t \mathrm{Tr}(F''(\zeta_s)(u_s, u_s)) \; ds\]
in $L^1(\O)$ as $n\to\infty$.

\smallskip
\textit{Step 2}: One has
\begin{align*}
\Big| &\sum_{i=0}^{2^n-1} F'(\zeta(t_i^n)) \Big(\int_{t_i^n}^{t_{i+1}^n} v(s) \; ds\Big) - \int_0^t F'(\zeta_s)v(s) \;ds \Big|\\
 &\leq \sum_{i=0}^{2^n-1} \int_{t_i^n}^{t_{i+1}^n} |(F'(\zeta(t_i^n))-F'(\zeta_s))v(s)|\;ds \leq \sup_{|s-r|\leq t2^{-n}} \|F'(\zeta_s) - F'(\zeta_r)\| \int_0^t \|v(s)\|\;ds.
 \end{align*}
By the the pathwise continuity of $\zeta$ and the dominated convergence theorem the latter converges to zero in $L^1(\O)$ as $n\to \infty$.

\smallskip
\textit{Step 3}: As in Lemma \ref{lem:estDzeta} one can show that $\zeta(t) \in \DD^{1,2}(E)$ for each $t\in [0,T]$. Therefore,
by Proposition \ref{ChainRule-2}, we have $F'(\zeta(t_i^n)) \in \DD^{1,2}(E^*)$ for all $i,n$.
By Proposition \ref{delta-continuity-thrm} one has $u\in {\rm Dom}(\delta)$, and with Lemma \ref{delta-intbyparts}, we obtain
\begin{align*}
\sum_{i=0}^{2^n-1} F'(\zeta(t_i^n))\Big(\int_{t_i^n}^{t_{i+1}^n} u(s) \;dW_U(s)\Big) & = \sum_{i=0}^{2^n-1} \int_{t_i^n}^{t_{i+1}^n} \inprod{u(s), F'(\zeta(t_i^n))}\;dW_U(s) \\
& \ \  + \sum_{i=0}^{2^n-1} \int_{t_i^n}^{t_{i+1}^n} \inprod{u(s), F''(\zeta(t_i^n))D(\zeta(t_i^n))(s)}_{\mathrm{Tr}}\, ds.
\end{align*}

By Lemma \ref{Ito-step4-formula} with $Z(s) := F''(\zeta(s))$ one obtains
\begin{align*}
\sum_{i=0}^{2^n-1} \int_{t_i^n}^{t_{i+1}^n} \inprod{u(s), F''(\zeta(t_i^n))D(\zeta(t_i^n))(s)}_{\mathrm{Tr}}\, ds \to \int_0^t \inprod{u(s), F''(\zeta(s))((D^-\zeta)(s))}_{\mathrm{Tr}} \, ds
\end{align*}
in $L^1(\O)$, as $n\to\infty$.
To finish the proof we need to show that
\begin{align*}
\int_0^t \sum_{i=0}^{2^n-1}  \one_{(t_i^n,t_{i+1}^n)}(s) \inprod{u(s), F'(\zeta(t_i^n))}\;dW_U(s) \to \int_0^t \inprod{u(s), F'(\zeta(s))} \;dW_U(s),
\end{align*}
in $L^{1}(\O)$ as $n\to \infty$.
To prove the latter note that because of the identity in the Taylor development in the beginning of the proof, and the convergence in $L^1(\O)$ of all other terms, we know that the lefthand side of the previous formula converges in $L^1(\O)$ to some $\xi$, and it remains to identify its limit.
Since $\delta$ is a closed operator on $L^1(\O;H)$, it suffices to note that
\[\limn \Big\|s\mapsto \sum_{i=0}^{2^n-1}  \one_{(t_i^n,t_i^n)}(s) \inprod{u(s), F'(\zeta(t_i^n))} - \inprod{u(s), F'(\zeta(s))} \Big\|_{L^1(\O;H)}=0.\]
where we used the pathwise continuty of $\zeta$.
We can conclude that $\xi = \int_0^t \inprod{u(s), F'(\zeta(s))}\, d W_U(s)$ and this completes the proof.
\end{proof}

\end{document}